\documentclass[11pt,twoside,reqno]{amsart}
\usepackage{amsxtra}
\usepackage{color}
\usepackage{amsopn}
\usepackage{amsmath,amsthm,amssymb,arydshln}
\usepackage{mathrsfs,mathtools}
\usepackage{stmaryrd}
\usepackage{hyperref}
\usepackage{enumerate}
\usepackage{xcolor}
\usepackage{float}
\usepackage{pifont}
\usepackage{adjustbox}

\newtheorem{theorem}{Theorem}[section]
\newtheorem{corollary}[theorem]{Corollary}
\newtheorem{proposition}[theorem]{Proposition}
\newtheorem{lemma}[theorem]{Lemma}
\newtheorem*{theorem*}{Theorem}

\theoremstyle{definition}

\theoremstyle{remark}
\newtheorem{remark}[theorem]{Remark}

\makeatletter
\@namedef{subjclassname@2020}{%
  \textup{2020} Mathematics Subject Classification}
\makeatother

\newcommand{\R}{{\mathbb R}}
\renewcommand{\H}{\mathrm{H}}

\newcommand{\Z}{\mathbb Z}

\newcommand{\beq}{\begin{equation}}
\newcommand{\eeq}{\end{equation}}
\renewcommand{\a}{\alpha}
\renewcommand{\b}{\beta}

\newcommand{\f}{\varphi}
\newcommand{\g}{\gamma}

\renewcommand{\l}{\lambda}
\newcommand{\m}{\mu}
\renewcommand{\o}{\omega}

\renewcommand{\t}{\tau}

\renewcommand{\L}{\Lambda}

\newcommand{\psip}{\psi_+}
\newcommand{\psim}{\psi_-}

\newcommand{\tpsi}{\widetilde{\psi}}
\newcommand{\tw}{\widetilde{w}}
\newcommand{\tf}{\widetilde{\varphi}}

\newcommand{\U}{{\mathrm U}}
\newcommand{\SU}{{\mathrm{SU}}}

\newcommand{\SO}{{\mathrm {SO}}}
\newcommand{\Sp}{{\mathrm {Sp}}}
\newcommand{\GL}{{\mathrm {GL}}}

\newcommand{\G}{{\mathrm G}}

\newcommand{\K}{{\mathrm K}}

\newcommand{\Spin}{\mathrm{Spin}}

\newcommand{\W}{\wedge}

\DeclareMathOperator\End{End}

\DeclareMathOperator\Ad{Ad}
\DeclareMathOperator\ad{ad}

\DeclareMathOperator\vol{vol}

\newcommand{\Scal}{{\rm Scal}}

\newcommand{\frg}{\mathfrak{g}}

\newcommand{\frh}{\mathfrak{h}}

\newcommand{\frm}{\mathfrak{m}}
\newcommand{\frn}{\mathfrak{n}}

\newcommand{\frsu}{\mathfrak{su}}

\newcommand{\st}{\ |\ }

\newcommand{\diag}{{\rm diag}}

\newcommand{\sst}{\scriptscriptstyle}

\newcommand{\la}{\langle}
\newcommand{\ra}{\rangle}

\numberwithin{equation}{section}

\title[The twisted G$_{2}$ equation for strong G$_{2}$-structures with torsion]{The twisted G$_{\mathbf2}$ equation 
for strong G$_{\mathbf2}$-structures with torsion}

\author{Anna Fino} 
\address{{\scriptsize Dipartimento di Matematica ``G.~Peano'' \\ Universit\`a degli Studi di Torino\\
Via Carlo Alberto 10\\
10123 Torino\\ Italy\\
 and Department of Mathematics and Statistics\\
 Florida International University\\
 Miami  FL 33199, USA}}
\email{annamaria.fino@unito.it, afino@fiu.edu}

\author{Luc\'ia  Mart\'in-Merch\'an}
\address{{\scriptsize Department of Pure Mathematics \\ University of Waterloo \\
200 University Avenue West\\
N2L 3G1 Waterloo, Ontario \\ Canada}}
\email{lucia.martinmerchan@uwaterloo.ca}

\author{Alberto Raffero}
\address{{\scriptsize Dipartimento di Matematica ``G.~Peano'' \\ Universit\`a degli Studi di Torino\\
Via Carlo Alberto 10\\
10123 Torino\\ Italy}}
\email{alberto.raffero@unito.it}


\begin{document}
\begin{abstract} 
We discuss general properties of strong G$_2$-structures with torsion and we investigate the 
twisted G$_2$ equation, which represents the G$_2$-analogue of the twisted Calabi-Yau equation for SU$(n)$-structures introduced in \cite{GRST}.
In particular, we show that invariant strong G$_2$-structures with torsion do not occur on compact non-flat solvmanifolds. 
This implies the non-existence of non-trivial solutions to the twisted Calabi-Yau equation on 
compact solvmanifolds of dimensions $4$ and $6$. 
More generally, we prove that a compact, connected homogeneous space admitting invariant strong G$_2$-structures  
with torsion is diffeomorphic either to $S^3 \times T^4$  or to $S^3 \times S^3 \times S^1$, up to a covering, 
and that in both cases solutions to the twisted G$_2$ equation exist. 
Finally, we discuss the behavior of the homogeneous Laplacian coflow for strong G$_2$-structures with torsion on these spaces. 
\end{abstract}
\maketitle

%
\section{Introduction}
On a Riemannian $n$-manifold $(M, g)$, a metric connection $\nabla$ is said to have totally skew-symmetric torsion if the $3$-covariant tensor 
$T$ obtained from its torsion tensor by lowering the contravariant index is skew-symmetric. 
In such a case, $\nabla$ is related to the Levi-Civita connection $\nabla^{g}$ of $g$ as follows
\[
\nabla_X Y  =   \nabla^{g}_X Y +  \frac 12\,g^{-1}T( X , Y), \quad X, Y \in \Gamma (TM). 
\]
Metric connections with totally skew-symmetric torsion are object of interest in 
complex non-K\"ahler geometry \cite{Bis,Gau} and generalized geometry \cite{Gar,FS,Str}, among others. 
Moreover, they naturally occur in theoretical and mathematical physics, e.g.~in Type II string theory and in supergravity theories \cite{Agr,FI,Iva,IP}. 

\smallskip

When $M$ is oriented and the structure group of its frame bundle admits a reduction to a closed subgroup $\G\subset\mathrm{SO}(n)$, 
the existence of a $\G$-connection with totally skew-symmetric torsion can be characterized 
in terms of the intrinsic torsion of the $\G$-structure, see \cite[Prop.~4.1]{FI}. 
In particular, by \cite[Thm.~4.7]{FI}, 
a 7-manifold $M$ with a $\G_2$-structure admits a $\G_2$-connection with totally skew-symmetric torsion $\nabla$ if and only if 
the definite $3$-form $\f$ defining the $\G_2$-structure satisfies the equation 
\begin{equation}\label{IntG2}
d \star \varphi = \theta \wedge \star \varphi,
\end{equation}
where $\theta\in\Omega^1(M)$ is the {\em Lee form} of the $\G_2$-structure 
and $\star$ is the Hodge operator defined by the Riemannian metric $g_{\varphi}$ 
and the orientation induced by $\varphi$. In this case, the connection $\nabla$ is unique.  

In the literature, $\G_2$-structures satisfying \eqref{IntG2} are called  
{\em$\G_2$-structures with torsion} ({\em$\G_2T$-structures} for short) or {\em integrable}, albeit the condition \eqref{IntG2} does not correspond 
to a reduction of the Riemannian holonomy.  
Since $\nabla\f=0$, one also has $\nabla g_\f=0,$ namely $\nabla$ is a metric connection with totally skew-symmetric torsion. 

The torsion $3$-form $T$ of $\nabla$ can be written in terms of $\f$ as follows
\begin{equation}\label{eqn:tor}
T= \frac{1}{6} {\star( d\varphi\W \varphi)}  \varphi - \star d\varphi + \star(\theta \wedge \varphi), 
\end{equation}
where $\theta = -\tfrac13\star(\star d\f\W\f)$. 
In particular, it vanishes identically if and only if the $\G_2$-structure is {\em parallel}, namely $\nabla^{g_\f}\f=0$ or, equivalently, 
$d\f=0$ and $d\star\f=0$ \cite{FerGray}. In analogy with strong K\"ahler with torsion Hermitian structures, 
$\G_2T$-structures with closed torsion $3$-form are called {\em strong} $\G_2T$-structures \cite{ChSw}. 

A special type of $\G_2T$-structures is given by coclosed ones, i.e., those satisfying the condition $d \star \varphi =0$ or, equivalently, 
$\theta=0$. Seven-dimensional manifolds admitting coclosed $\G_2$-structures with parallel torsion $3$-form were studied in \cite{Fr}.

\smallskip

In high-energy physics, $\G_2T$-structures are related to the Hull-Strominger system of equations in dimension seven \cite{CFT,FI,FI2,Iva}. 
Given a $7$-dimensional compact spin manifold $M$ and a principal $\G$-bundle $P \to M,$ 
the {\em Hull-Strominger system} for a $\G_2$-structure $\f$ on $M,$ 
a smooth function $f \in \mathcal{C}^\infty(M)$, a connection $\nabla$ on $TM$ and a connection $A$ on $P$ can be written as follows
\[
\begin{split}
d\varphi \wedge \varphi &= 0, \quad  d\star \varphi=\, -4 df \wedge \star \varphi,\\
F_A \wedge \star \varphi & = 0,\quad R_\nabla \wedge \star \varphi= 0,\\
d\star(- d\varphi -4df \wedge \varphi ) &= \frac{\alpha'}{4}\left(\mathrm{tr}(F_A\wedge F_A) - \mathrm{tr}({R_\nabla \wedge R_\nabla})\right), 
\end{split}
\]
where $F_A$ and $R_\nabla$ are the curvatures of $A$ and $\nabla$, respectively, and $\alpha'$ is a positive constant  
(the square of the string length). 
Thus, any definite $3$-form $\f$ solving this system is a $\G_2T$-structure, it satisfies the additional constraint $d\varphi \wedge \varphi=0$, and 
its Lee form $\theta = -4df$ is exact. Moreover, the torsion $3$-form $T= -\star(d\varphi +4df \wedge \varphi)$ is constrained by the last equation 
of the system, which is related to the anomaly cancellation condition in string theory. 

As shown in \cite{GMPW} (see also \cite[Prop.~3.1]{CFT}), in the special case where $M$ is compact and 
the torsion $3$-form $T$ is closed, any solution to the system
\[
d\varphi \wedge \varphi=0, \qquad
d\star \varphi= -4 df\wedge \star \varphi,  \qquad 
dT= 0, \quad 
\]
must have constant $f$ and $T=0$. Consequently, the $\G_2$-structure must be parallel.  
This is no longer true if one only requires the Lee form to be closed. This leads to the study of the following system of equations
\begin{equation} \label{eqn:strom7}
d\varphi \wedge \varphi=0, \quad
d\star \varphi= \theta \wedge \star \varphi, \quad 
dT= 0, \quad
d\theta=0. 
\end{equation} 
We shall refer to \eqref{eqn:strom7} as the {\em twisted $\G_2$ equation}, 
since it represents the $\G_2$-analogue of the {\em twisted Calabi-Yau equation} for $\SU(n)$-structures introduced in \cite[Def.~2.3]{GRST}.  
The latter is defined as follows for an $\SU(n)$-structure with fundamental $2$-form $\omega$, integrable complex structure $J$,  
complex volume form $\Psi$ and Lee form $\theta_\o = - J\delta\o$  on a $2n$-dimensional manifold
\[
d\Psi = \theta_\o\W\Psi,\quad dd^c\o=0,\quad d\theta_\o = 0. 
\]
By \cite[Prop.~2.10]{GRST}, a compact complex surface admitting solutions to the twisted Calabi-Yau equation must be one of the following:   
a flat torus, a $K3$ surface with a K\"ahler Ricci-flat metric (so that $d\o=0$ and $d\Psi=0$), or a quaternionic Hopf surface. 
Moreover, in \cite[Sect.~2.4]{GRST}, solutions to this equation are shown to exist on a class of non-K\"ahler complex threefolds 
diffeomorphic to $S^3 \times T^3$ whose universal covering is the product 
of the universal covering $\mathbb{C}^2\smallsetminus\{0\}$ of a quaternionic Hopf surface and the complex line $\mathbb{C}$.  
In the context of $(0,2)$ mirror symmetry for compact non-K\"ahler complex manifolds, 
solutions to the twisted Calabi-Yau equation on the homogeneous Hopf surface $\SU(2)\times\U(1)$  were recently used in \cite{AAF} to 
construct the first known examples of $(0,2)$ mirror pairs.

\smallskip 
 
In this paper, we discuss general properties of strong $\G_2T$-structures and we investigate their existence and the existence 
of solutions to the twisted $\G_2$ equation on compact $7$-manifolds. 

In Section \ref{sec:IntClosed}, we determine some constraints imposed by the existence of strong $\G_2T$-structures on a $7$-manifold. 
In particular, we use a constraint on the scalar curvature determined in Proposition \ref{prop:scalar} to prove that 
every left-invariant strong $\G_2T$-structure on a unimodular solvable Lie group is parallel and induces a flat metric  (Corollary \ref{G2TUnimod}). 
This implies the non-existence of compact solvmanifolds admitting invariant non-parallel strong $\G_2T$-structures (Corollary \ref{CorSolvG2T}), 
generalizing an analogous result for compact nilmanifolds previously obtained by Chiossi-Swann \cite{ChSw}.  
Moreover, we show that strong $\G_2T$-structures solving the twisted $\G_2$ equation \eqref{eqn:strom7} 
naturally exist on the product of the torus $T^{7-2n}$, $n=2,3$,  and 
a $2n$-dimensional manifold admitting an $\SU(n)$-structure solving the twisted Calabi-Yau equation (Proposition \ref{TCYG2link}). 
In this case, the Lee form of the strong $\G_2T$-structure coincides with the Lee form $\theta_\o$ of the $\SU(n)$-structure 
and the torsion $3$-form is $T=d^c\o$. 
As a consequence of these results, we get the non-existence of invariant non-trivial solutions to the twisted Calabi-Yau equation 
on compact $2n$-dimensional solvmanifolds, when $n=2,3$  (Proposition \ref{PropTCYSol}). 
This result in the case $n=2$ also follows from the classification result \cite[Prop.~2.10]{GRST} recalled above.

In Section \ref{SecProdS1}, we focus on the special case where $(M,g_\f)$ is the Riemannian product of a $6$-manifold $N$ and the circle $S^1$,   
so that $N$ is naturally endowed with an $\SU(3)$-structure $(\o,\Psi)$. 
In Proposition \ref{prop:product-closedT}, we characterize the strong $\G_2T$ condition \eqref{IntG2} and the twisted $\G_2$ equation \eqref{eqn:strom7} 
in terms of equations for $(\o,\Psi)$ in the case where the Lee form of the $\G_2T$-structure is proportional to a global closed $1$-form on $S^1$. 
Notice that this situation is complementary to the one considered in Proposition \ref{TCYG2link}, where the Lee form of the strong $\G_2T$-structure 
coincides with the Lee form of the $\SU(3)$-structure on the $6$-manifold.  
This characterization implies that the third Betti number of $N$ must be positive when $N$ is compact 
and $N\times S^1$ admits a strong $\G_2T$-structure fulfilling the previous requirements (Corollary \ref{Corb3N}).  
We also determine sufficient conditions on the $\SU(3)$-structure that ensure the existence of three different $\G_2$-structures on $N\times S^1$ sharing  
the same metric and torsion $3$-form: a strong $\G_2T$-structure, a strong $\G_2T$-structure solving the twisted $\G_2$ equation and a 
coclosed $\G_2$-structure (Proposition \ref{prop:product-str}). 
In this case, the $\G_2$-connection with totally skew-symmetric torsion associated with these $\G_2$-structures has zero Ricci tensor.

We then focus on the homogeneous setting. 
Recall that a $7$-manifold $M$ with a $\G_2$-structure $\f$ is said to be {\em homogeneous} if there exists a Lie group 
$\G$ acting transitively on $M$ and leaving $\f$ invariant. In such a case, $M$ is $\G$-equivariantly diffeomorphic to a quotient $\G/\H$, 
for some closed subgroup $\H\subset\G$.  
Compact homogeneous spaces admitting invariant $\G_2$-structures were classified in \cite{LM,Reid}. 
By \cite{Reid}, invariant coclosed $\G_2$-structures always occur on such spaces. 

Using the classification result of \cite{Reid} and our results on strong $\G_2T$-structures discussed in Sections \ref{sec:IntClosed} and \ref{SecProdS1}, 
we prove the following. 

\begin{theorem} 
Let $M$ be a seven-dimensional, compact, connected homogeneous space for the almost effective action of a connected Lie group $\G$.  
If $M$ admits invariant strong $\G_2T$-structures, then it is diffeomorphic either to $S^3 \times T^4$ or to $S^3 \times S^3 \times S^1$, 
up to a covering. 
\end{theorem} 

The proof is obtained studying separately the case where the homogeneous space $M=\G/\H$  is irreducible, namely it 
is not covered by the product of lower-dimensional homogeneous spaces (Section \ref{sec:Irreducible}),   
and the case where $M=\G/\H$ is reducible (Section \ref{sec:Reducible}). 
In the former case, we show that invariant $\G_2T$-structures occur on various spaces, but none of them has closed torsion. 
As a byproduct of our discussion, we describe the space of invariant $\G_2T$-structures on all Aloff-Wallach spaces $N^{p,q}= \SU(3)/\U(1)_{p,q}$, 
where $p,q\in\Z$ are the weights of the action of $\U(1)$ on $\SU(3)$, and on the $7$-sphere $S^7= \Sp(2)/\Sp(1)$. 

On the spaces $S^3 \times T^4\cong\SU(2)\times\U(1)^4$  and  $S^3 \times S^3 \times S^1\cong \SU(2)\times \SU(2)\times\U(1)$, 
we also obtain homogeneous examples of strong $\G_2T$-structures 
solving the twisted $\G_2$ equation \eqref{eqn:strom7} and whose associated $\G_2$-connection with totally skew symmetric torsion is flat 
(see propositions \ref{S3T4G2T} and \ref{S3S3S1G2T}). 
In the first case, the example is obtained starting with the homogeneous Hopf manifold $\SU(2)\times \U(1) \cong S^3\times S^1$ endowed 
with a left-invariant $\SU(2)$-structure solving the twisted Calabi-Yau equation and applying the construction of Proposition \ref{TCYG2link}. 
In the second case, we construct a left-invariant $\SU(3)$-structure on the homogeneous space $\SU(2)\times\SU(2)\cong S^3\times S^3$ 
which does not solve the twisted Calabi-Yau equation, but satisfies the conditions determined in Proposition \ref{prop:product-closedT} and 
thus induces a left-invariant strong  $\G_2T$-structure on $\SU(2)\times\SU(2)\times\U(1)$. 
Since the $\SU(3)$-structure satisfies the hypothesis of Proposition \ref{prop:product-str}, too, 
the existence of a solution to the twisted $\G_2$ equation follows. 
It is an open problem to see whether there exist compact $7$-manifolds admitting a strong $\G_2T$-structure inducing a non-flat 
$\G_2$-connection with totally skew-symmetric torsion.

Finally, in Section \ref{CoflowSect}, we consider the homogeneous Laplacian coflow for strong $\G_2T$-structures
on $\SU(2)\times\U(1)^4$ and on $\SU(2)\times\SU(2)\times\U(1)$. 
The Laplacian coflow for coclosed $\G_2$-structures was introduced in \cite{KMT} as 
an analogue of Bryant's Laplacian flow for closed $\G_2$-structures \cite{Bry}, 
and it prescribes the evolution of a one-parameter family of coclosed $\G_2$-structures $\psi_t = \star_t\f_t$, i.e., $d\psi_t=0$, 
in the direction of its Hodge Laplacian 
\begin{equation}\label{CoflowIntro}
\partial_t\psi_t = \Delta_{\psi(t)}\psi(t). 
\end{equation}
In this equation, $\Delta_\psi = \delta_{g_\psi}d+d\delta_{g_\psi}$ denotes the Hodge Laplacian 
determined by the metric $g_\psi$ induced by the definite $4$-form $\psi$ and a given orientation that is fixed throughout the flow.  
Although it is currently not known whether this flow is well-posed on compact manifolds, solutions are known to exist in some specific cases, 
see, for instance, \cite{BF,KMT,KL,LSS,PS}. 
Since the $\G_2T$ condition \eqref{IntG2} is an equation for the $4$-form $\psi=\star\f$ generalizing the condition $d\psi=0$, 
it seems worth understanding whether the flow \eqref{CoflowIntro} represents a useful tool to study the larger class of $\G_2T$-structures. 
Here, as a preliminary step in this direction, we study the Laplacian coflow starting at the examples of strong $\G_2T$-structures obtained in 
Section \ref{sec:Reducible}. 
It turns out that the behavior of the flow is similar in both situations: the solution $\psi(t)$ exists (at least) for all positive times 
and it defines a strong $\G_2T$-structure.   
Moreover, the $7$-manifold splits as the Riemannian product of an associative submanifold $Y^3$ and a coassociative submanifold $X^4$.  
In detail, $Y^3$ is either a $3$-torus or a $3$-sphere that 
is calibrated by $\f(t)=\star_t\psi(t)$ and whose volume increases along the flow, 
while $X^4=\SU(2)\times\U(1)$ is calibrated by $\psi(t)$ and its volume is fixed along the flow. 
These results provide some insights on the problem and pave the way towards a more detailed study of the Laplacian coflow for (strong) 
$\G_2T$-structures.

%

\section{Definite forms in dimension seven}\label{SecDefForm7}

Let $V$ be a real vector space of dimension $n$ and 
consider the action of the general linear group $\GL(V )$ on the space $\Lambda^k V^*$.  
An element  $\rho \in \Lambda^k V^*$ is said to be {\em stable} if its $\GL(V )$-orbit is open in  $\Lambda^kV^*$.   

Let us focus on the case $\dim V = 7$. 
Every 3-form $\phi\in\Lambda^3V^*$ gives rise to a symmetric bilinear map 
\[
b_\phi:V \times V \rightarrow\Lambda^7V^*,\quad (v,w) \mapsto \frac16\, \iota_v\phi\W\iota_w\phi\W\phi.
\] 
By \cite{Hitchin}, the $\GL(V)$-orbit of $\phi$ is open in $\Lambda^3V^*$ if and only if $\det(b_\phi)^{1/9}\in\Lambda^7V^*$ is nonzero. 

A stable 3-form $\phi$ on $V$ is said to be {\em definite} if the symmetric bilinear form 
\[
g_\phi:V\times V \rightarrow\R,\quad g_\phi\coloneqq \det(b_\phi)^{-1/9}b_\phi,
\] 
is positive definite. The $\GL(V)$-stabilizer of a definite $3$-form is isomorphic to the exceptional Lie group $\G_2$. 
A definite $3$-form $\phi$ induces both an inner product $g_\phi$ and an orientation on $V,$ and thus a Hodge operator $\star$.  
In particular, one has $dV_\phi \coloneqq \vol_{g_\phi} = \det(b_\phi)^{1/9}$. 
Moreover, there exists a $g_\phi$-orthonormal basis $\mathcal{B}=(e_1,\ldots,e_7)$ of $V$ with dual basis $\mathcal{B}^*=(e^1,\ldots,e^7)$ 
such that 
\[
\phi = e^{123}+ e^1\wedge (e^{45}+e^{67}) + e^2\wedge (e^{46}-e^{57}) - e^3\wedge (e^{47}+e^{56}), 
\]
where $e^{ij\cdots}$ is a shortening for the wedge product of covectors $e^i \W e^j \W\cdots$. 
Both $\mathcal{B}$ and $\mathcal{B}^*$ are usually called {\em adapted bases} to $\phi$. 

\smallskip
Let $\Lambda^3_+V^*$ denote the set of definite 3-forms on $V.$ This is an open subset of $\Lambda^3V^*$. 
The map 
\[
\Lambda^3_+V^* \rightarrow \Lambda^4V^*,\quad \phi \mapsto \star\phi,
\]
is a double covering onto an open subset $\Lambda^4_+V^*\subset \Lambda^4V^*$ of stable $4$-forms.  
The forms belonging to $\Lambda^4_+V^*$ are called {\em definite}, too, and their $\GL(V)$-stabilizer is isomorphic to 
$\G_2 \cup \left\{-\mathrm{Id}_V\cdot \G_2\right\}$. 
In particular, every definite $4$-form gives rise to an inner product on $V,$ but not to an orientation. 
 
 \smallskip
 
The inner product determined by a definite $4$-form can be computed as follows. 
Choosing a volume form $\Omega$ on $V,$ one has an isomorphism  $\Lambda^4V^* \cong \L^3 V \otimes \L^7 V^*$ that identifies 
$\rho\in\Lambda^4V^*$ with the element $\hat{\rho}\otimes \Omega\in \L^3 V \otimes \L^7 V^*$ such that $\iota_{\hat{\rho}}\Omega=\rho$. 
Applying the construction of the previous case, one obtains a symmetric bilinear map  
\[
B_\rho : V^* \times V^* \to \L^7 V \otimes (\Lambda^7 V^*)^{\otimes 3}\cong (\L^7 V^*)^{\otimes 2},\quad 
(\a, \b) \mapsto \left(\frac16\iota_{\a}  \hat\rho \wedge  \iota_{\beta} \hat\rho \wedge \hat\rho\right)\otimes\Omega^{\otimes3}, 
\]
and one has that $\rho$ is stable if and only if $|\det(B_{\rho})|^{1/12} \in \L^7 V^*$ is nonzero. 
Moreover, $\rho$ is definite if the symmetric bilinear form 
\[
\overline{\beta}_\rho:V^*\times V^*\rightarrow \R,\quad \overline{\beta}_\rho\coloneqq |\det(B_{\rho})|^{-1/6} B_{\rho},
\]
is definite. The inner product induced by $\rho$ is then $\overline{g}_\rho = \overline{\beta}_\rho$ if $\overline{\beta}_\rho$ is positive definite, and 
$\overline{g}_\rho = -\overline{\beta}_\rho$ otherwise. 
When $\phi\in\Lambda^3_+V^*$, so that $\star\phi\in\Lambda^4_+V^*$, then one has $\overline{g}_{\star\phi} = g_\phi^{-1}$.

\smallskip

A $7$-manifold $M$ admits $\G_2$-structures if and only if it is orientable and spin. 
In such a case, it is possible to consider the open subbundle $\Lambda^3_+T^*M$ of $\Lambda^3T^*M$ 
whose fibre over each point $x$ of $M$ is the set $\Lambda^3_+T_x^*M$ of definite $3$-forms on $T_xM$. 
A $\G_2$-structure on $M$ is then given by a definite $3$-form $\f\in\Omega^3_+(M)\coloneqq \Gamma(\Lambda^3_+T^*M)$. 
As explained above, this $3$-form gives rise  to a Riemannian metric $g_\f$ and an orientation on $M.$

%

\section{Strong $\G_2T$-structures}\label{sec:IntClosed}

Let $(M,\f)$ be a $7$-manifold with a $\G_2$-structure.  
It is well-known, see e.g.~\cite{Bry}, that the following $g_\f$-orthogonal decompositions hold
\[
\Omega^2(M) = \Omega^2_7 \oplus\Omega^2_{14},\quad
\Omega^3(M) = \mathcal{C}^\infty(M)\,\f \oplus \Omega^3_7 \oplus \Omega^3_{27},
\]
where 
\[
\begin{split}
\Omega^2_7 		&= \left\{\sigma\in\Omega^2(M) \st \sigma\W\f=2\star\sigma \right\},\quad
\Omega^2_{14}		= \left\{\sigma\in\Omega^2(M) \st \sigma\W\f=-\star\sigma \right\},\\
\Omega^3_7	& = \left\{\star(\alpha\W\f) \st \alpha\in\Omega^1(M)\right\},\quad
\Omega^3_{27}	= \left\{\rho\in\Omega^3(M) \st \rho\W\f=0,~\rho\W\star\f=0\right\}. 
\end{split}
\]
The Hodge operator allows one to obtain the analogous decompositions for $\Omega^k(M) = \star\Omega^{7-k}(M)$, for $k=4,5$. 
According to these decompositions and by \cite[Prop.~1]{Bry}, there exist unique differential forms  $\tau_k \in \Omega^k(M)$, $k=0,1,2,3$, such that
\begin{equation}\label{TorsionForms}
\begin{split}
d\varphi 		&= \t_0 \star \varphi + 3\, \t_1 \wedge \varphi + \star \t_3, \\
d\star \varphi	&= 4\, \t_1 \wedge \star \varphi + \star \t_2, 
\end{split}
\end{equation}
where $\tau_2\in\Omega^2_{14}$ and $\tau_3\in\Omega^3_{27}$. 
The forms $\tau_0$, $\tau_1$, $\tau_2$ and $\tau_3$ are the {\em intrinsic torsion forms} of the $\G_2$-structure $\f$.  
These forms are all zero if and only if the intrinsic torsion of the $\G_2$-structure vanishes, namely if and only if $\f$ is both closed and coclosed.  
This is equivalent to $\f$ being parallel with respect to the Levi Civita connection of $g_\f$ \cite{FerGray}. 
When this happens, the Riemannian metric $g_\f$ has zero Ricci tensor, and the $\G_2$-structure is said to be {\em torsion-free} or {\em parallel}.

Comparing the second equation in \eqref{TorsionForms} with \eqref{IntG2}, one has that  $\f$ is a $\G_2T$-structure if and only if $\tau_2=0$. 
In such a case, the Lee form is $\theta = 4\tau_1$ and 
the expression of the torsion $3$-form $T$ given in \eqref{eqn:tor} can be rewritten in terms of the  intrinsic torsion forms as follows 
\begin{equation}\label{Ttorforms}
T= \frac{1}{6} \tau_0 \, \varphi +\star(\tau_1\W\f) - \tau_3.
\end{equation}
This is the decomposition of $T$ with respect to the $g_\f$-orthogonal splitting of $\Omega^3(M)$. 
In particular, $T$ is zero if and only if $\tau_0$, $\tau_1$ and $\tau_3$ are all zero, namely if and only if $\f$ is parallel. 

As observed in \cite{FI2}, using the general identity 
\begin{equation}\label{1formid}
\star(\alpha\W\f)\W\f = -4\star\alpha,\quad \forall~\alpha\in\Omega^1(M), 
\end{equation}
one obtains
\begin{equation}\label{LeeT}
\theta = 4\tau_1 =-\star(T\W\f).
\end{equation}
Notice that the Lee form is zero if and only if the $\G_2$-structure $\f$ is coclosed.  

From now on, we shall restrict our attention to strong $\G_2T$-structures, namely to the case where $dT=0$.

\begin{proposition}\label{prop:scalar}
Let $\f$ be a strong $\G_2T$-structure. Then, the scalar curvature of $g_\f$ is 
\[
\Scal(g_\f)= 10 \delta \t_1 + \frac{49}{24}\t_0^2 + 24|\t_1|^2.
\]

\end{proposition}
\begin{proof}
We first prove the identity 
\begin{equation}\label{eqnormtau3}
|\tau_3|^2 = \frac76\,\tau_0^2+12|\tau_1|^2+4\delta\tau_1.
\end{equation}
Since $\theta = 4\tau_1$, from equation \eqref{LeeT} we obtain
\[
4\,d\star\tau_1 	= -d\left( T\W\f\right) = T\W d\f 
			= \left(\frac76\,\tau_0^2+12|\tau_1|^2 -|\tau_3|^2 \right)dV_\f,
\]
where we used the expression of $T$ given in \eqref{Ttorforms}, the expression of $d\f$ given in \eqref{TorsionForms}, the identity \eqref{1formid}, 
and the identity $|\tau_1\W\f|^2 = 4|\tau_1|^2$. 
The identity \eqref{eqnormtau3} then follows. 

According to \cite[eq.~(4.28)]{Bry}, the scalar curvature of the metric induced by a $\G_2T$-structure is
\[
\Scal(g_\f) = 12 \delta \t_1 + \frac{21}{8}\t_0^2 + 30 |\t_1|^2 - \frac{1}{2}  |\t_3|^2.
\]
Taking into account the formula \eqref{eqnormtau3}, the expression of $\Scal(g_\f)$ given in the statement follows.  
\end{proof}

The previous proposition provides a constraint for the existence of (non-parallel) strong $\G_2T$-structures on certain $7$-manifolds. 
For instance, we have the following. 
\begin{corollary}\label{G2TUnimod}
Let $\f$ be a left-invariant strong $\G_2T$-structure on a $7$-dimensional unimodular Lie group $\G$. 
Then, $g_\f$ has non-negative scalar curvature. Moreover, $g_\f$ is flat whenever $\G$ is solvable or $\Scal(g_\f)=0$. 
When this happens, the $\G_2$-structure $\f$ is parallel. 
\end{corollary}
\begin{proof}
Since the $7$-dimensional Lie group $\G$ is unimodular, every left-invariant $6$-form on it is closed (cf., e.g., \cite[Ch.~4, Lemma 1.1]{Schulte}). 
We then have $\delta\tau_1 = -\star d \star \tau_1 = 0$, whence it follows that 
\[
\Scal(g_\f) = \frac{49}{24}\t_0^2 + 24|\t_1|^2 \geq 0. 
\]
If $\Scal(g_\f)=0$, we have $\tau_0=0$, $\tau_1=0$ and thus $\tau_3=0$ by \eqref{eqnormtau3}. 
Therefore, the $\G_2$-structure $\f$ is parallel and the associated metric $g_\f$ is Ricci-flat and thus flat by \cite{AK}.  
If $\G$ is solvable, then $g_\f$ must be flat by \cite[Thm.~3.1]{Mil} and thus $\Scal(g_\f)=0$. 
\end{proof}

Since every simply connected Lie group admitting co-compact discrete subgroups must be unimodular \cite[Lemma 6.2]{Mil}, 
we obtain the following.  
\begin{corollary}\label{CorSolvG2T}
A seven-dimensional compact solvmanifold cannot admit invariant non-parallel strong $\G_2T$-structures. 
\end{corollary}
\begin{proof} 
Recall that a compact solvmanifold is given by the quotient $\Gamma\backslash\G$ of a simply connected solvable Lie group $\G$ 
by a co-compact discrete subgroup $\Gamma \subset \G$, 
and that an invariant $\G_2$-structure on $\Gamma\backslash\G$ is induced by a left-invariant one on $\G$. 
Since $\G$ is unimodular, the thesis follows from Corollary \ref{G2TUnimod}. 
\end{proof}

This corollary extends a previous result by Chiossi and Swann \cite{ChSw}, 
who proved the non-existence of non-parallel strong $\G_2T$-structures on the product of the circle $S^1$ 
with a six-dimensional nilmanifold.

\smallskip

When a strong $\G_2T$-structure is coclosed and the intrinsic torsion form $\tau_0$ is constant, 
a further constraint for the existence of such structures on a compact $7$-manifold holds. 
Indeed, the manifold must have non-zero third Betti number. 
\begin{proposition}\label{lem:co-closed}
Let  $\varphi$ be a coclosed strong $\G_2T$-structure on a compact $7$-manifold $M,$ and assume that  
$\tau_0$ is constant. Then, $T$ is harmonic and $b_3(M)>0$.
\end{proposition}
\begin{proof}
Since $\star T= \frac{7}{6} \tau_0 \star \varphi - d\varphi$ and $\tau_0$ is constant, we have that $T$ is coclosed and thus harmonic. 
If $\f$ is not torsion-free, then $T$ is not zero and $b_3(M)>0$. Otherwise, $\f$ is a harmonic $3$-form on $M$ and thus $b_3(M)>0$. 
\end{proof}

We conclude this section with some remarks on the interplay between the twisted $\G_2$ equation \eqref{eqn:strom7} for $\G_2$-structures 
and the twisted Calabi-Yau equation for $\SU(n)$-structures introduced in \cite{GRST}. 

Recall that an $\SU(n)$-structure on a $2n$-dimensional manifold $N^{2n}$ is defined by the data of an almost Hermitian structure $(g,J)$ with  
fundamental $2$-form $\o = g(J\cdot,\cdot)$ and a complex $(n,0)$-form $\Psi$ of unit norm. 
We let $\psip\coloneqq \mathrm{Re}(\Psi)$ and $\psim\coloneqq \mathrm{Im}(\Psi)$, and we denote by $\star_{2n}$ the Hodge operator 
induced by the metric $g$ and the volume form $dV_g = \frac{1}{n!}\o^n$. 

The Lee form of an $\SU(n)$-structure is the $1$-form $\theta_\o \coloneqq - J\delta_g\o$ and it can be characterized as 
the unique $1$-form such that $d\o^{n-1}=\theta_\o\W\o^{n-1}$. 
Here, we use the convention $J\alpha = \alpha(J\cdot,\ldots,J\cdot)$, for $\alpha\in\Omega^k(N^{2n})$. 

An $\SU(n)$-structure $(\o,\Psi)$ is said to be {\em torsion-free} if $d\omega=0$ and $d\Psi=0$. In such a case, the corresponding metric $g$ is Ricci-flat. 

An $\SU(n)$-structure $(\o,\Psi)$ solves the {\em twisted Calabi-Yau equation} if the following hold
\begin{equation}\label{TCY}
d\Psi = \theta_\o\W\Psi,\quad dd^c\o=0,\quad d\theta_\o = 0. 
\end{equation}
In particular, $J$ is integrable and the Hermitian structure $(g,J)$ is strong K\"ahler with torsion. 
Clearly, a torsion-free $\SU(n)$-structure trivially satisfies these equations. 

The next result points out a link between $\SU(n)$-structures solving the twisted Calabi-Yau equation \eqref{TCY}, for $n=2,3$, 
and $\G_2$-structures solving the twisted $\G_2$ equation \eqref{eqn:strom7}. 
\begin{proposition}\label{TCYG2link}
Let $N^{2n}$ be a $2n$-manifold, $n=2,3$, endowed with an $\SU(n)$-structure $(\o,\Psi)$ solving the twisted Calabi-Yau equation \eqref{TCY}. 
Then, the $7$-manifold $M = N^{2n}\times \R^{7-2n}$ admits a strong $\G_2T$-structure $\f$ solving the twisted $\G_2$ equation \eqref{eqn:strom7}, 
and whose Lee form $\theta$ and torsion $3$-form $T$ coincide with $\theta_\o$ and $d^c\o$, respectively. 
\end{proposition}
\begin{proof}
Let us start with the case $n=2$. On the manifold $M=N^4\times \R^3$, we consider the definite $3$-form 
\begin{equation}\label{G2N4T3}
\f = ds^1\W ds^2\W ds^3 + ds^1\W\o + ds^2\W \psip - ds^3\W \psim,
\end{equation}
$s^1,s^2,s^3$ being the standard coordinates on $\R^3$. 
This $3$-form induces the product metric $g_\f = \sum_{i=1}^3 ds^i\otimes ds^i + g$ 
and the orientation $dV_\f = ds^1\W ds^2 \W ds^3\W dV_g$. 
Since $\o$, $\psip$ and $\psim$ are self-dual with respect to the Hodge operator $\star_4$ on $N^4$ induced by 
the metric $g$ and the orientation $dV_g = \tfrac12\o^2$, we have
\[
\star \f = \tfrac12\o^2 + ds^2\W ds^3\W\o - ds^1\W ds^3\W\psip - ds^1\W ds^2\W\psim. 
\]
Since $d\o = \theta_\o\W\o$ and $d\Psi = \theta_\o\W\Psi$, we obtain $d\star \f = \theta_\o\W\star \f$. 
Thus, $\f$ is a $\G_2T$-structure with closed Lee form $\theta = \theta_\o$. Moreover, we have 
\[
d\f\W\f = \theta_\o \W\left(ds^1\W\o + ds^2\W \psip - ds^3\W\psim\right)\W\f = 0, 
\]
since the self-dual forms $\o$, $\psip$ and $\psim$ are pairwise orthogonal. 
Finally, the torsion $3$-form of $\f$ is 
\[
T = \star( \theta\W\f - d\f ) = \star (\theta_\o\W ds^1\W ds^2\W ds^3 ) = - \star_4\theta_\o  = Jd\o = d^c\o, 
\]
since $\theta_\o = J\star_4d\star_4\o = \star_4Jd\o$ and $J\o=\o$. 

Let us now focus on the case $n=3$. Here we consider the $\G_2$-structure on $M = N^6\times \R$ given by the definite $3$-form
\begin{equation}\label{G2N6T1}
\f = \o \W ds + \psip,
\end{equation}
$s$ being the standard coordinate on $\R$. Then, $g_\f = g + ds^2$, $dV_\f = dV_g\W ds$ and $\star \f = \tfrac12\o^2 + \psim \W ds$. 
The identity $d\star \f = \theta_\o\W\star \f$ immediately follows from $d\o^2=\theta_\o\W\o^2$ and $d\psip = \theta_\o\W \psip$. 
Moreover, since $\o\W\psip = 0$ and $d\psip=\theta_\o\W\psip$, we have
\[
d\f\W\f 	= \left( d\o\W ds+\theta_\o\W\psip\right)\W \left(\o\W ds+\psip\right) 
		= -d\o\W\psip \W ds = \o\W\theta_\o\W\psip \W ds = 0. 
\]
The differential of the fundamental $2$-form is given by $d\o =  w_3 +  \tfrac12\theta_\o\W\o$, where $w_3\in\Omega^3(N^6)$ satisfies 
$w_3\W\omega = 0$, $w_3\W\psi_{\pm}=0$ and $Jw_3 = \star_6w_3$ (cf.~\cite{BeVe} and see Section \ref{SecProdS1} for more details). 
We then have
\[
d^c\o = Jd\o = Jw_3 + \tfrac12 J\theta_\o\W\o   = \star_6\left( w_3 -\tfrac12 \theta_\o\W\o \right). 
\]
where we used the identity $J\beta\W\o = -\star_6(\beta\W\o)$, which holds for every $1$-form $\beta$ on $N^6$. 
Finally, we obtain 
\[
T =  \star( \theta\W\f-d\f) = \star((\theta_\o\W\o-d\o)\W ds) = \star_6(d\o -\theta_\o\W\o) =  \star_6\left(w_3 - \tfrac12 \theta_\o\W\o\right) = d^c\o. 
\]
\end{proof}

As a consequence of Proposition \ref{TCYG2link} and Corollary \ref{G2TUnimod}, we obtain the following non-existence result for 
non-trivial solutions of the twisted Calabi-Yau equation on compact solvmanifolds. 
\begin{proposition}\label{PropTCYSol}
For $n=2,3$, every invariant $\SU(n)$-structure solving the 
twisted Calabi-Yau equation on a $2n$-dimensional compact solvmanifold $\Gamma\backslash \G$ must be torsion-free. 
\end{proposition} 
\begin{proof}
An invariant $\SU(n)$-structure on $\Gamma\backslash \G$ is induced by a left-invariant one on the unimodular simply connected solvable Lie group $\G$. 
By Proposition \ref{TCYG2link},  if $(\omega,\Psi)$ is a left-invariant $\SU(n)$-structure on $\G$ solving the twisted Calabi-Yau equation, 
then $\G\times \R^{7-2n}$ admits a strong $\G_2T$-structure $\f$ defined as in \eqref{G2N4T3} when $n=2$ 
and as in \eqref{G2N6T1} when $n=3$. In both cases, the $\G_2$-structure is clearly left-invariant. 
By Corollary \ref{G2TUnimod}, $\f$ must be parallel and the corresponding metric $g_\f$ must be flat. 
Since both $\f$ and $\star\f$ are closed, we see that $d\omega=0$ and $d\Psi=0$. 
The thesis then follows. 
\end{proof}

\begin{remark}
The previous result in the case $n=2$ also follows from the classification result \cite[Prop.~2.10]{GRST}. 
\end{remark}

%

\section{Strong $\G_2T$-structures on Riemannian products $N\times S^1$}\label{SecProdS1}
In the next sections, we will investigate the existence of non-parallel strong $\G_2T$-structures on compact homogeneous spaces. 
As we will see, in various examples the $7$-manifold $(M,g_\f)$ is isometric to a Riemannian product $(N\times S^1,g+\eta\otimes\eta)$, where 
$(N,g)$ is a six-dimensional Riemannian manifold and $\eta$ is a closed $1$-form on $S^1$. 
In this case, there exists an SU(3)-structure $(\o, \Psi)$ 
on $N$ with corresponding metric $g$ such that 
\[
\f = \omega\W\eta+\psip. 
\] 
Moreover, one has $g_\f=g+\eta\otimes\eta$, $dV_\f = dV_g\W\eta$ and 
\[
\star\f = \frac12\omega\W\omega+\psim\W\eta, 
\]
where $\psim=\star_6\psip$ and $\omega^2 = 2\star_6\omega$, 
Recall that $\omega$ and $\psi_\pm$ satisfy the compatibility condition $\omega\W\psi_\pm=0$ and the normalization condition 
$3\,\psip\W\psim = 2\,\omega^3 = 12\,dV_g$.

\smallskip

The intrinsic torsion of an $\SU(3)$-structure  $(\omega, \Psi)$ is determined by the unique differential forms 
$w_1^\pm\in\mathcal{C}^\infty(N)$, $w_2^\pm\in\Omega^2_8(N)\coloneqq\left\{\sigma\in\Omega^2(N) \st J\sigma=\sigma,~\sigma\W\omega^2=0 \right\}$, 
$w_3\in\Omega^3_{12}(N) \coloneqq \left\{\gamma\in\Omega^3(N) \st \gamma\W\omega=0,~\gamma\W\psi_{\pm}=0 \right\}$, 
$w_4, w_5 \in\Omega^1(N)$ such that 
\[
\begin{split}
d\omega	&= -\frac{3}{2}w_1^- \psi_+ + \frac{3}{2}w_1^+ \psi_- + w_3 + w_4 \wedge \omega ,\\
d\psi_+	&= w_1^+ \omega^2 +\star_6  w_2^+ +  \psi_+\W w_5,\\
d\psi_-	&= w_1^- \omega^2 + \star_6 w_2^- +  \psi_+ \W Jw_5,
\end{split}
\] 
see \cite{BeVe,ChSa}. 
These forms are the {\em intrinsic torsion forms} of the SU(3)-structure. When they all vanish, one has $d\omega=0$ and $d\Psi=0$, 
and thus the SU(3)-structure is torsion-free.

\begin{remark}\label{RemTCYSU3}
From the expression of $d\o$, we see that $\tfrac12 d \o^2 = d\o\W\o = w_4\W\o^2$. 
Thus, the Lee form of an $\SU(3)$-structure is given by $\theta_\o= 2 w_4$. 
It is known that the almost complex structure $J$ is integrable if and only if $w_1^\pm=0$ and $w_2^\pm=0$, see \cite{ChSa}. 
In particular, the only (possibly) non-vanishing intrinsic torsion forms of an $\SU(3)$-structure solving the twisted Calabi-Yau equation \eqref{TCY} are 
$w_3$ and $2w_4 = -w_5 = \theta_\o$. 
\end{remark}

The intrinsic torsion forms of the $\G_2$-structure $\f = \omega\W\eta+\psip$ on $M=N\times S^1$ are related to the intrinsic torsion 
forms of the SU(3)-structure $(\omega,\Psi)$. The complete description of this interplay is given in \cite[Thm.~3.1]{ChSa}. 
Decomposing the Lee form of $\f$ as follows 
\[
\theta = \beta + \lambda\eta, 
\]
with $\beta\in\Omega^1(N)$ and $\lambda\in\mathcal{C}^\infty(N)$, one has that $\f$ is a $\G_2T$-structure, i.e., $\tau_2=0$, if and only if 
\begin{equation}\label{G2Tprod}
w_1^- = \frac{\lambda}{2},\quad w_2^-=0,\quad w_5=-2w_4=-\beta,
\end{equation}
see \cite{ChSw}. In particular, $\beta=\theta_\o$. 
Moreover, the torsion $3$-form is given by
\begin{equation}\label{Tprod}
T = \frac12 w_1^+\psip + \frac12 w_1^- \psim +\star_6 w_3 -\frac12\star_6(\theta_\o\W\omega) - w_2^+\W\eta, 
\end{equation}
as one can easily compute starting from the general expression \eqref{eqn:tor} 
and using the interplay between the Hodge operators $\star$ and $\star_6$ 
(see also \cite{ChSw}, but pay attention to the presence of an unnecessary extra summand which does not affect the computations afterwards).  

\smallskip

We now focus on the case where $\theta =  \lambda\eta$, namely $\theta_\o=0$. Clearly, this is not covered by Proposition \ref{TCYG2link}. 
Notice that $\theta$ is closed if and only if $\lambda$ is constant 
and that $\lambda=0$ if and only if the $\G_2$-structure $\f=\omega\W\eta+\psip$ is coclosed. 
We summarize some useful properties in the next result.

\begin{proposition} \label{prop:product-closedT}
Let $(N,\omega,\Psi)$ be a $6$-manifold with an $\SU(3)$-structure, consider the $\G_2$-structure
$\varphi= \omega \wedge \eta + \psi_+$ on $M=N\times S^1$, and assume that its Lee form is given by $\theta =  \lambda\eta$, 
for some $\lambda\in\mathcal{C}^\infty(N)$. 
Then
\begin{enumerate}[$(1)$]
\item\label{prodclosed1}  $\varphi$ is a $\G_2T$-structure if and only if 
\begin{equation}\label{G2TNS}
w_1^- = \frac{\lambda}{2},\quad  w_2^-=0,\quad  w_4 = 0 = w_5, 
\end{equation}
and its torsion $3$-form $T$ is closed if and only if 
\begin{equation}\label{G2TNSdT}
dw_2^+ = 0, \quad  d\star_6 w_3 = -\frac12 d\left( w_1^+\psip +  w_1^- \psim\right).
\end{equation}
In such a case, $\lambda = 2w_1^-$ is constant and the Lee form $\theta$ is closed. 
In particular, when $\lambda=0$, $\f$ is a coclosed $\G_2$-structure with closed torsion $3$-form. 
Furthermore, when $N$ is compact, the intrinsic torsion form $w_2^+$ must vanish. 
\item\label{prodclosed2} 
If conditions \eqref{G2TNS} and \eqref{G2TNSdT} hold and $N$ is compact, then the intrinsic torsion form  $w_1^+$ is constant. 
Moreover, \eqref{G2TNSdT} is equivalent to
\[
\delta_g w_3 =  \left( \left( w_1^+\right)^2 + \left( w_1^- \right)^2\right) \omega,
\]
and $T$ is a harmonic $3$-form on $N.$ 
\item\label{prodclosed3}  The $\G_2$-structure $\f=\omega \wedge \eta + \psi_+$ is a solution of the 
twisted $\G_2$ equation \eqref{eqn:strom7} if and only if the 
conditions \eqref{G2TNS} and \eqref{G2TNSdT} together with $w_1^+=0$ hold. 
If $N$ is compact and $\f$ is not coclosed, then  
\eqref{G2TNSdT} is equivalent to 
\[
\delta_g w_3 =   \left( w_1^- \right)^2 \omega,
\]
and $T$ is a non-zero harmonic $3$-form on $N.$ 
Otherwise, both 
the $\SU(3)$-structure $(\omega,\Psi)$ and the $\G_2$-structure $\f$ are torsion-free. 
\end{enumerate}
\end{proposition}
\begin{proof}
\begin{enumerate}[$(1)$]
\item Since $\theta =  \lambda\eta$, we have $\beta=0$ and the conditions characterizing a strong $\G_2T$-structure 
$\varphi= \omega \wedge \eta + \psi_+$ 
with Lee form $\theta$ immediately follow from \eqref{G2Tprod} and \eqref{Tprod}. 
In particular, the SU(3)-structure $(\omega,\Psi)$ on $N$ satisfies 
\begin{equation}\label{SU3sG2T}
d\omega	= -\frac{3}{4}\lambda\, \psi_+ + \frac{3}{2}\,w_1^+ \psi_- + w_3,\quad
d\psi_+	= w_1^+ \omega^2 +\star_6  w_2^+,\quad
d\psi_-	=  \frac{\lambda}{2}\, \omega^2.
\end{equation}
From the first equation, we see that $d\omega\W\omega=0$. Taking the exterior derivative of both sides of the third equation we get 
$d\lambda\W\omega^2=0$, whence it follows that $\lambda = 2w_1^-$ is constant and $\theta$ is closed. 
If $N$ is compact, the condition $dw_2^+=0$ implies $w_2^+=0$. Indeed, the norm of $w_2^+$ with respect to the $L^2$ inner product 
induced by $g$ on $\Lambda^2T^*N$ is 
\[
\left\| w_2^+\right\|^2  = \int_N w_2^+\W \star_6 w_2^+ =  \int_N w_2^+\W d\psip = -  \int_N dw_2^+\W \psip = 0, 
\]
where we used the  identity $w_2^+\W\omega^2=0$. 
\item Assume that \eqref{G2TNS} and \eqref{G2TNSdT} hold and that $N$ is compact. 
Then, $w_2^+=0$ and taking the exterior derivative of both sides of the second equation in \eqref{SU3sG2T}, we obtain  
$0 = dw_1^+ \W \omega^2,$ from which the first claim follows. 
The first equation in \eqref{G2TNSdT} is then trivially satisfied, while the second one becomes 
\[
d\star_6 w_3 =  -\frac12 \left(w_1^+d\psip + w_1^- d \psim  \right)
=  -\frac12  \left( \left( w_1^+\right)^2 + \left( w_1^- \right)^2\right)  \omega^2, 
\]
whence we get $\delta_g w_3 = -\star_6 d\star_6 w_3 =   \left( \left( w_1^+\right)^2 + \left( w_1^- \right)^2\right)  \omega$. 
The expression \eqref{Tprod} of the torsion $3$-form reduces to
\[
T = \frac12 w_1^+\psip + \frac12 w_1^- \psim +\star_6 w_3, 
\]
and thus $T$ is a closed $3$-form on $N$. Using equations \eqref{SU3sG2T} with $w_2^+=0$, we then obtain
\[
d\star_6T =  \frac12 w_1^+d\psim - \frac12 w_1^- d\psip - d w_3 = 0. 
\]
\item 
Since 
\[
d\f \W \f = (-d\omega \W \psip + d\psip \W \omega )\wedge \eta  = \left(\frac{3}{2}w_1^+ \psip \W \psim + w_1^+ \omega^3 \right)\wedge \eta 
= 12 w_1^+dV_\f,
\] 
a solution of the twisted $\G_2$ equation \eqref{eqn:strom7} must satisfy \eqref{G2TNS}, \eqref{G2TNSdT}, 
and the additional constraint $w_1^+=0$.  
Assume now that $N$ is compact, so that $w_2^+=0$. 
If $\f$ is not coclosed, then $w_1^-\neq 0$ and $T$ is a non-zero harmonic 3-form on $N$ by part (\ref{prodclosed2}). 
Otherwise, $\psip$ and $\psim$ are closed forms, $d\omega=w_3$, and \eqref{G2TNSdT} becomes $\delta_gw_3=0$. 
Therefore $w_3$ must vanish, so that both $(\omega,\Psi)$ and $\f$ are torsion-free. 
\end{enumerate}
\end{proof}

\begin{corollary}\label{Corb3N}
Let $N$ be a compact $6$-manifold with an $\SU(3)$-structure $(\omega,\Psi)$ inducing 
a strong $\G_2T$-structure $\f=\omega\W\eta+\psip$  with Lee form $\theta = \lambda\eta$ on $M=N\times S^1$. 
Then, $N$ (and thus $M$) has positive third Betti number. 
\end{corollary}
\begin{proof}
If $\f$ is not torsion-free, then $T$ is a non-zero harmonic $3$-form on $N$ by Proposition \ref{prop:product-closedT}. 
Otherwise, $(\omega,\Psi)$ is torsion-free, $\psip$ is harmonic and the thesis follows. 
\end{proof}

Proposition \ref{prop:product-closedT} singles out a class of SU(3)-structures defined by the conditions 
$w_1^+=0$, $w_2^\pm=0$, $w_4=0$ and $w_5=0$. 
These SU(3)-structures are known as {\em double half-flat} in the literature and they were investigated, for instance, in \cite{Schulte}. 
The next result was partly observed in \cite{ChSw}. 
\begin{corollary}\label{corDHF}
Let $(N,  \omega,\Psi)$ be a $6$-manifold with a  double half-flat $\SU(3)$-structure. 
Then, the $\G_2$-structure $\varphi= \omega \wedge e^7 + \psi_+$  on $N \times S^1$ 
is a solution of the twisted $\G_2$ equation \eqref{eqn:strom7} if and only if $\Delta_g \omega =  4\left(w_1^-\right)^2 \omega$.
\end{corollary}
\begin{proof}
For every double half-flat SU(3)-structure $(\omega,\Psi)$, the function $w_1^-$ is constant, 
since $d\psim=w_1^-\omega^2$ and $d\omega\W\omega=0$. Moreover, the equations \eqref{G2TNSdT} reduce to the single equation
\[
\delta_g w_3 =  \left(w_1^-\right)^2 \omega.
\]
The conclusion follows then from Proposition \ref{prop:product-closedT} and the identity 
\[
\Delta_g \omega	= - \star_6 d \star_6 d\omega = \frac{3}{2}w_1^- \star_6 d \star_6  \psi_+ - \star_6 d \star_6w_3 
				= 3\left(w_1^-\right)^2 \omega +\delta_g w_3. 
\]
\end{proof}

Under suitable assumptions, it is possible to modify the SU(3)-structure $(\omega,\Psi)$ on $N$ 
in order to obtain a new one whose corresponding $\G_2$-structure on $N\times S^1$ shares the same metric and torsion $3$-form 
with $\f=\omega\wedge\eta+\psip$, but satisfies different properties. 

\begin{proposition}\label{prop:product-str}
Let $N$ be a $6$-manifold with an $\SU(3)$-structure $(\omega,\Psi)$ satisfying \eqref{G2TNS} and \eqref{G2TNSdT} 
with $w_2^+=0$ and $(w_1^+,w_1^-)\neq(0,0)$, 
and let $\varphi= \omega \wedge \eta + \psi_+$ be the induced strong $\G_2T$-structure with Lee form $\theta=2w_1^-\eta$ on $M=N\times S^1$.  
Then,  $M$ admits both a $\G_2$-structure solving the twisted $\G_2$ equation \eqref{eqn:strom7} and a coclosed $\G_2$-structure whose 
associated metric and (harmonic) torsion $3$-form coincide with those associated with $\varphi$. 
Moreover, the $\G_2$-connection with totally skew-symmetric torsion corresponding to $(g_\f,T)$ has zero Ricci tensor. 
\end{proposition}

\begin{proof}
By Proposition \ref{prop:product-closedT} and the assumption on the intrinsic torsion forms of $(\omega,\Psi)$, we have 
\[
d\omega	= - \frac{3}{2}\,w_1^- \psi_+ + \frac{3}{2}\,w_1^+ \psi_- + w_3,\quad
d\psi_+	= w_1^+ \omega^2 ,\quad
d\psi_-	= w_1^-\omega^2,
\]
and 
\[
\delta_g w_3 = \left( \left( w_1^+\right)^2 + \left( w_1^- \right)^2\right) \omega. 
\] 
In particular,  $w_1^\pm$ are constant, $w_3$ is closed, and the torsion $3$-form of $\f$ 
\[
T = \frac12 w_1^+\psip + \frac12 w_1^- \psim +\star_6 w_3
\]
is harmonic. 

Let $g$ be the metric corresponding to the  SU(3)-structure  $(\omega,\Psi)$. 
Recall that replacing $\psip$ with the $3$-form $\tpsi_+ = a\,\psip+b\,\psim$, where $a^2+b^2=1$, 
one obtains an SU(3)-structure $(\omega,\widetilde{\Psi})$ inducing the same metric $g$  
and such that $\tpsi_- =\star_6\tpsi_+ = a\,\psim-b\,\psip$. 
We will denote the intrinsic torsion forms of $(\omega,\widetilde{\Psi})$ by $\tw_1^\pm, \tw_2^\pm,\tw_3,\tw_4,\tw_5$.  

First, we let
\[
\tpsi_{+} = \frac{1}{\sqrt{\left(w_1^+\right)^2 + \left(w_1^-\right)^2}} \left(w_1^-\psi_+ - w_1^+ \psi_- \right).
\]
Then, a direct computation shows that
\[
\tw_1^- = \sqrt{\left(w_1^+\right)^2 + \left(w_1^-\right)^2},\quad \tw_3=w_3,
\]
and that the remaining intrinsic torsion forms of $(\omega,\widetilde{\Psi})$ are zero. Moreover,  
\[
\delta_g \tw_3 = \delta_g w_3 = \left(\left(w_1^-\right)^2 + \left(w_1^+\right)^2\right)\omega =  \left(\tw_1^-\right)^2\omega. 
\] 
Therefore, by (\ref{prodclosed3}) of Proposition \ref{prop:product-closedT}, 
$\tf = \omega\W\eta+\tpsi_+$ is a strong $\G_2T$-structure on $N\times S^1$ with Lee form $\widetilde\theta= 2\tw_1^- \eta$ 
solving the twisted $\G_2$ equation \eqref{eqn:strom7}. 
Finally, we have  $g_{\tf} = g + \eta\otimes \eta = g_\f$, and the torsion $3$-form corresponding to $\tf$ is 
\[
\widetilde{T}=  \frac{1}{2} \tw_1^-\tpsi_- + \star_6 \tw_3 =  \frac{1}{2}w_1^+\psi_+ +\frac12 w_1^- \psi_- + \star_6 w_3,  
\] 
and thus it coincides with $T$. 

As for the coclosed $\G_2$-structure, it is sufficient choosing 
\[
\tpsi_{+} = \frac{1}{\sqrt{\left(w_1^+\right)^2 + \left(w_1^-\right)^2}} \left(w_1^+\psi_+ + w_1^- \psi_- \right),
\]
so that the non-zero intrinsic torsion forms are $\tw_1^+ = \sqrt{\left(w_1^+\right)^2 + \left(w_1^-\right)^2}$ and $\tw_3=w_3$. Then,  
$\tf$ is a coclosed $\G_2$-structure such that $g_{\tf}=g_\f$ and with harmonic torsion $3$-form 
\[
\widetilde{T}=  \frac{1}{2} \tw_1^+\tpsi_+ + \star_6 \tw_3 =  \frac{1}{2}w_1^+\psi_+ +\frac12 w_1^- \psi_- + \star_6 w_3 = T. 
\] 
By \cite[Thm.~5.4]{FI}, the $\G_2$-connection with totally skew symmetric torsion $T$ corresponding to a coclosed $\G_2$-structure 
has zero Ricci tensor if and only if $T$ is harmonic, hence the last claim follows. 
\end{proof}

\begin{remark} 
Recall from (\ref{prodclosed1}) of Proposition \ref{prop:product-closedT} that the condition $w_2^+=0$ always holds when $N$ is compact and 
the $\SU(3)$-structure satisfies \eqref{G2TNS} and \eqref{G2TNSdT}. 
\end{remark}

%
\section{Irreducible compact homogeneous spaces}\label{sec:Irreducible}

Let $M=\G/\H$ be a seven-dimensional compact, connected homogeneous space 
with respect to the almost effective action of a connected Lie group $\G$ and assume that $\H$ is connected. 
By \cite[Thm.~1]{Reid}, if $M$ admits $\G$-invariant $\G_2$-structures and it is {\em irreducible}, namely it is not the product of a circle 
and another homogeneous space, then $M$ must be one of the spaces listed in Table \ref{TabIrred}, up to a covering. 
In Table \ref{TabIrred}, $N^{p,q}$ denotes an Aloff-Wallach space, $B^7$ denotes the Berger space and $V^{5,2}$ denotes the Stiefel manifold of 
orthonormal pairs in $\R^5$. 

The spaces listed in Table \ref{TabIrred} are simply connected and have vanishing third Betti number. 
In particular, the Lee form of an invariant $\G_2$-structure on any of them cannot be closed unless it vanishes. 
Furthermore, if an invariant $\G_2T$-structure with vanishing Lee form exists, then it cannot be strong by Proposition \ref{lem:co-closed}. 
Consequently, it is not possible to obtain invariant solutions to the twisted $\G_2$ equation \eqref{eqn:strom7}.

\begin{table}[h!]
\centering
\renewcommand\arraystretch{1.2}
\begin{tabular}{|c|c|c|c|}  
\hline
& $\G$& $\H$& $M=\G/\H$ \\ \hline 
1&$\SU(3)$& $\U(1)$ & $N^{p,q},~{p,q \in \Z,~ p\geq q\geq0,~\mbox{gcd}(p,q)=1}$ \\ \hline
2&$\SU(3)\times \U(1)$& $\U(1)^2$ & $N^{p,q},~{p,q \in \Z,~p\geq q\geq0,~\mbox{gcd}(p,q)=1}$ \\ \hline
3&$\SU(3)\times  \SU(2)$& $\SU(2) \times \U(1)$ & $N^{1,1}$ \\ \hline
4&$\Sp(2)$& $\Sp(1)$ & $S^7$ \\ \hline
5&$\SU(4)$& $\SU(3)$ & $S^7$ \\ \hline
6&$\Sp(2)\times \U(1)$& $\Sp(1)\times \U(1)$ & $S^7$ \\ \hline
7&$\Sp(2)\times \Sp(1)$& $\Sp(1)\times \Sp(1)$ & $S^7$ \\ \hline
8&$\Spin(7)$& $\G_2$ & $S^7$ \\ \hline
9&$\SO(5)$& $\SO(3)$ & $B^7$ \\ \hline
10&$\SO(5)$& $\SO(3)$ & $V^{5,2}$ \\ \hline
11&$\SU(2)^3$& $\U(1)^2$ & $Q^{1,1,1}$ \\ \hline
12&$\SU(3)\times \SU(2)$& $\SU(2)  \times \U(1)$ & $M^{1,1,0}$ \\ \hline
\end{tabular}  \vspace{0.2cm}
\caption{Irreducible compact homogeneous spaces admitting invariant $\G_2$-structures, up to a covering \cite{Reid}. }\label{TabIrred}
\end{table}

We now prove that none of these spaces admits invariant strong $\G_2T$-structures. 
We first deal with cases $1$-$7$ of Table \ref{TabIrred},  
and then we focus on cases $8$-$12$.  
For each space $M=\G/\H$, we will consider a reductive decomposition $\frg = \frm \oplus \frh$ of the Lie algebra of $\G$, 
where the $\Ad(\H)$-invariant subspace $\frm$ is identified with the tangent space to $M$ at the identity coset, 
and we will identify $\G$-invariant tensors on $M$ with $\Ad(\H)$-invariant tensors on $\frm$.

\subsection*{Cases 1-7 of Table \ref{TabIrred}} 
We first rule out the Aloff-Wallach spaces $N^{p,q}$, with $p\neq q$, and then the Aloff-Wallach space $N^{1,1}$ 
and the $7$-sphere $\Sp(2)/\Sp(1)$. 
Once we are done, we can easily discard the cases $2$-$3$ and $5$-$7$. 

\subsubsection*{\underline{The Aloff-Wallach spaces $N^{1,0}$ and $N^{p,q}$ with $p>q>0$}}\ 

\smallskip
\noindent
Let $p,q\in\Z$. Recall that $N^{p,q}=\SU(3)/\U(1)_{p,q}$, where $\U(1)_{p,q}$ is diagonally embedded into $\SU(3)$ as 
\begin{equation}\label{AWUSUemb}
\U(1)_{p,q} = \left\{ \diag\left(z^p,z^q,z^{-p-q}\right) \st z \in \U(1)\right\} \subset \SU(3). 
\end{equation}
It is not restrictive assuming that $p\geq q\geq0$ and $\mbox{gcd}(p,q)=1$. 
There are three relevant situations that must be considered: the exceptional Aloff-Wallach spaces $N^{1,1}$ and $N^{1,0}$ and the 
generic Aloff-Wallach spaces $N^{p.q}$ with $p>q>0$. We begin with the generic case and the exceptional case $N^{1,0}$.

\begin{proposition}\label{PropAWpq}
The generic Aloff-Wallach space $N^{p,q}=\SU(3)/\U(1)_{p,q}$, with $p> q>0$, and the exceptional Aloff-Wallach space $N^{1,0}$ 
admit a $3$-parameter family of $\SU(3)$-invariant $\G_2T$-structures whose 
torsion $3$-form is never closed. 
\end{proposition}
\begin{proof}
Let $\G \coloneqq \SU(3)$ and $\H\coloneqq\U(1)_{p,q}$. 
We consider a reductive decomposition $\frg= \frm \oplus \frh$, with $\frm$ spanned by 
\begin{equation}\label{AWpqbase}
\begin{split}
& 
e_1= \frac13\, \diag\left(i(-p-2q),i(2p+q),i(q-p)\right), \\
&
e_2 = \begin{pmatrix}
0 & 1 & 0 \\
-1 & 0 & 0 \\
0 & 0 & 0 
\end{pmatrix},\quad
e_3=  \begin{pmatrix}
0 & i & 0 \\
i & 0 & 0 \\
0 & 0 & 0  
\end{pmatrix},\quad 
e_4 = \begin{pmatrix}
0 & 0 & 1 \\
0 & 0 & 0 \\
-1 & 0 & 0 
\end{pmatrix},\\
&
e_5 = \begin{pmatrix}
0 & 0 & i \\
0 & 0 & 0 \\
i & 0 & 0 
\end{pmatrix},\quad
e_6=\begin{pmatrix}
0 & 0 & 0 \\
0 & 0 & 1 \\
0 & -1 & 0 
\end{pmatrix},\quad
e_7=  \begin{pmatrix}
0 & 0 & 0 \\
0 & 0 & i \\
0 & i & 0  \end{pmatrix},
\end{split}
\end{equation}
and $\frh$ spanned by $e_8=\diag(ip,iq,-i(p+q))$. 
We see that $\ad(e_8)e_1=0$ and 
\[
\begin{split}
& \ad(e_8)e_2=(p-q)e_3,\quad \ad(e_8)e_3=(q-p)e_2,\quad \ad(e_8)e_4=(2p+q)e_5,\\ 
&\ad(e_8)e_5=-(2p+q)e_4,\quad \ad(e_8)e_6= (2q+p)e_7,\quad \ad(e_8)e_7=-(2q+p)e_6. 
\end{split}
\]
Thus, the decomposition of $\frm$ into $\Ad(\H)$-invariant irreducible summands is $\frm= \frm_0 \oplus \frm_1 \oplus \frm_2 \oplus \frm_3$, 
where $\frm_0 = \langle e_1 \rangle$ is a trivial submodule, and $\frm_1=\langle e_2,e_3\rangle$, $\frm_2=\langle e_4,e_5\rangle$ and 
$\frm_3=\langle e_6,e_7\rangle$ are non-trivial submodules. 
These modules are inequivalent when $p>q>0$, while $\frm_1\cong\frm_3$ when $p=1$ and $q=0$.

Let $(e^1,\dots,e^7)$ denote basis of $\frm^*$ dual to $(e_1,\ldots,e_7)$. By the Koszul formula, for every $x,y\in\frm$ we have
\[
de^i(x,y) = -e^i([x,y]_\frm), 
\]
and thus
\begin{equation}\label{deAWpq}
\begin{split}
de^1	&= \frac{3}{p^2+pq+q^2} \left((p+q)e^{23} + q e^{45}-p e^{67} \right), \\
de^2	&= -(p+q) e^{13}  + e^{46} + e^{57},\\
de^3	&= (p+q)e^{12} -e^{47} + e^{56}, \\
de^4	&= -q e^{15} - e^{26} + e^{37}, \\
de^5	&= q e^{14} - e^{27} - e^{36}, \\
de^6	&= p e^{17} + e^{24} +e^{35}, \\
de^7	&= -p e^{16} +e^{25} - e^{34}.
\end{split}
\end{equation}

Let us first focus on the case $p>q>0$. 
The space of $\Ad(\H)$-invariant $3$-forms on $\frm$ decomposes as follows
\[
(\Lambda^3\frm^*)^{\H} \cong \bigoplus_{i=1}^3\left(\frm_0^*\otimes \Lambda^2\frm_i^*\right) \oplus (\frm_1^*\otimes\frm_2^*\otimes\frm_3^*)^\H, 
\]
where the $2$-dimensional summand $ (\frm_1^*\otimes\frm_2^*\otimes\frm_3^*)^\H$ is spanned by 
\[
\gamma_1 \coloneqq e^{246} + e^{257} - e^{347} + e^{356} ,\quad \gamma_2 \coloneqq e^{247} - e^{256} + e^{346} +e^{357}.
\]
In particular, $\gamma_2 = d\left(e^{23}+e^{45}+e^{67}\right)$ is the unique $\Ad(\H)$-invariant closed (hence exact) 3-form on $\frm$. 
The generic $\Ad(\H)$-invariant $3$-form on $\frm$ is 
\[
\f = a_1 e^{123} + a_2 e^{145} + a_3 e^{167} + a_4\gamma_1+a_5\gamma_2,
\]
for some $a_1,\ldots,a_5\in\R$. We fix the orientation $e^{1234567}$, and we compute the symmetric bilinear form 
$b_\f: \frm \times \frm \rightarrow \Lambda^7\frm^*\cong \R$ introduced in Section \ref{SecDefForm7}, obtaining 
\[
\begin{split}
b_\f 	&= 	a_1a_2a_3 e^1 \odot e^1 \\
	&\quad- (a_4^2+a_5^2) \left[ a_1\left(e^2\odot e^2+e^3\odot e^3\right) 
		-a_2\left(e^4\odot e^4+e^5\odot e^5\right) + a_3\left(e^6\odot e^6+e^7\odot e^7\right)\right],
\end{split}
\]
where $e^i\odot e^j\coloneqq \frac12\left(e^i\otimes e^j + e^j\otimes e^i\right)$. 
From this, we see that $\f$ defines a $\G_2$-structure if and only if $a_4^2+a_5^2\neq0$ and either $a_1<0,a_3<0,a_2>0$ or $a_1>0,a_3>0,a_2<0$. 
When these conditions hold, we can compute the metric $g_\f$ induced by $\f$ 
and the $4$-form $\star\f$. We obtain 
\[
g_\f = (\det b_\f)^{-1/9}b_\f = 	\left(a_1 a_2 a_3\right)^{-1/3}\left(a_4^2 + a_5^2\right)^{-2/3}b_\f, 
\]
and
\[
\begin{split}
\star\f &= 	\left(a_1 a_2 a_3\right)^{1/3}\left(a_4^2 + a_5^2\right)^{2/3}\left(\frac{1}{a_1}\, e^{4567}  
		+\frac{1}{a_2}\, e^{2367}  
		+\frac{1}{a_3}\, e^{2345}  \right)\\
	&\quad +\left(\frac{a_1 a_2 a_3}{a_4^2 + a_5^2}\right)^{1/3} \left( 
		a_4 e^1\W\gamma_2 -a_5 e^1\W\gamma_1 \right). 
\end{split}
\]
 
We now look for (strong) $\G_2T$-structures belonging to this $5$-parameter family of $\Ad(\H)$-invariant $\G_2$-structures.
The Lee form must necessarily be $\theta = \mu e^1$, for some $\mu\in\R\smallsetminus\{0\}$. 
Using \eqref{deAWpq}, we obtain 
\[
\begin{split}
d\star\f  &= a_5 \left(\frac{a_1 a_2 a_3}{a_4^2 + a_5^2}\right)^{1/3}  e^1\W d\gamma_1 = 4 a_5 \left(\frac{a_1 a_2 a_3}{a_4^2 + a_5^2}\right)^{1/3} \left(e^{12367} - e^{12345} - e^{14567}\right),\\
\theta\W\star\f &= \mu \left(a_1 a_2 a_3\right)^{1/3}\left(a_4^2 + a_5^2\right)^{2/3} 
\left( \frac{1}{a_1}\, e^{14567} + \frac{1}{a_2}\,e^{12367} +\frac{1}{a_3}\, e^{12345}\right),
\end{split} 
\] 
Consequently, $\f$ is a $\G_2T$-structure with non-zero Lee form if and only if
\begin{equation}\label{G2TAWpq}
a_1 = -a_2 = a_3,\quad \mu = -4\frac{a_1a_5}{a_4^2+a_5^2}, \quad a_5\neq0. 
\end{equation}
Thus, there is a $3$-parameter family of $\Ad(\H)$-invariant $\G_2T$-structures on $\frm$ 
depending on two non-zero real parameters $a_1$, $a_5$ and on a real parameter $a_4$.  
The intrinsic torsion forms $\tau_0$ and $\tau_1$ of these structures are the 
following  
\[
\tau_0 = \frac17 \star(d\f\W\f) = \frac{24}{7}\frac{a_4}{(a_4^2+a_5^2)^{2/3}},\qquad 
\tau_1 = \frac14\, \theta = -\frac{a_1a_5}{a_4^2+a_5^2}\,e^1,
\]
and the torsion form $\tau_3$ can be obtained from $d\f$ using the first equation in \eqref{TorsionForms}. 

We can now compute the torsion $3$-form $T$ of $\f$ using the expression \eqref{Ttorforms}. We have
\begin{equation}\label{TAWpq}
T = 	\frac{3\,a_1^2}{(p^2+pq+q^2)\left(a_4^2+a_5^2\right)^{2/3}}\left( (p+q) e^{123} + q e^{145} -p e^{167} \right) 
	+\left(a_4^2+a_5^2\right)^{1/3}\,\gamma_1,
\end{equation}
which is clearly not closed for any choice of the parameters $a_1\neq0$, $a_5\neq0$ and $a_4\in\R$. 

In the case when $p=1$ and $q=0$, the space of $\Ad(\H)$ invariant $3$-forms on $\frm$ is 
\[
(\Lambda^3\frm^*)^{\H} \cong \bigoplus_{i=1}^3\left(\frm_0^*\otimes \Lambda^2\frm_i^*\right) \oplus (\frm_1^*\otimes\frm_2^*\otimes\frm_3^*)^\H\oplus \frm_0\otimes(\frm_1\otimes\frm_3)^{\H}, 
\]
where the additional summand $ \frm_0\otimes(\frm_1\otimes\frm_3)^{\H}$ is spanned by the closed $3$-forms 
$\gamma_3 \coloneqq e^{126}+e^{137}$ and $\gamma_4 \coloneqq e^{127}-e^{136}$. Thus, the generic $\Ad(\H)$-invariant $3$-form is 
\[
\f = a_1 e^{123} + a_2 e^{145} + a_3 e^{167} + a_4\gamma_1+a_5\gamma_2 + a_6\gamma_3+a_7\gamma_4. 
\]
Now, a similar discussion as above shows that $\f$ defines a $\G_2T$-structure if and only if $a_6=0=a_7$ and the conditions 
\eqref{G2TAWpq} are satisfied. With analogous computations, we obtain that the torsion $3$-form $T$ is given by \eqref{TAWpq}
with $p=1$ and $q=0$. In particular, it is not closed. 
\end{proof}

\subsubsection*{\underline{The exceptional Aloff-Wallach space $N^{1,1}$ and the $7$-sphere $S^7=\Sp(2)/\Sp(1)$ }}\ \vspace{0.2cm}  

\noindent We now show that neither the exceptional Aloff-Wallach space 
$N^{1,1} = \SU(3)/\U(1)_{1,1}$ nor the $7$-sphere $S^7=\Sp(2)/\Sp(1)$ admit invariant strong $\G_2T$-structures. 
Since these spaces share some similarities, we will investigate them simultaneously using a unified notation that we are going to introduce. 

\smallskip

Consider the exceptional Aloff-Wallach space $N^{1,1} = \SU(3)/\U(1)_{1,1}$, 
where the embedding $\U(1)_{1,1}\subset \SU(3)$ is described in \eqref{AWUSUemb} 
for $p=q=1$, and let $\G=\SU(3)$ and $\H=\U(1)_{1,1}$. 
We choose a reductive decomposition $\frg= \frm \oplus \frh$  
with $\frh= \la e_8 = \diag(i,i,-2i) \ra$ and 
$\frm$ spanned by
\[
\begin{split}
& 
e_1= \diag\left(i,-i,0\right), \\
&
e_2 = \begin{pmatrix}
0 & -1 & 0 \\
1 & 0 & 0 \\
0 & 0 & 0 
\end{pmatrix},\quad
e_3=  \begin{pmatrix}
0 & -i & 0 \\
-i & 0 & 0 \\
0 & 0 & 0  
\end{pmatrix},\quad 
e_4 = \begin{pmatrix}
0 & 0 & 1 \\
0 & 0 & 0 \\
-1 & 0 & 0 
\end{pmatrix},\\
&
e_5 = \begin{pmatrix}
0 & 0 & 0 \\
0 & 0 & -i \\
0 & -i & 0 
\end{pmatrix},\quad
e_6=\begin{pmatrix}
0 & 0 & 0 \\
0 & 0 & -1 \\
0 & 1 & 0 
\end{pmatrix},\quad
e_7=  \begin{pmatrix}
0 & 0 & i \\
0 & 0 & 0 \\
i & 0 & 0  \end{pmatrix},
\end{split}
\]
(notice that this is obtained from \eqref{AWpqbase} via an obvious basis change).  
We denote by $(e^1,\dots,e^7)$ the dual basis of $\frm^*$.

We have the $\Ad(\H)$-invariant decomposition $\frm= \frm_3 \oplus \frm_4$, where $\H$ acts trivially on 
$\frm_3= \la e_1, e_2, e_3 \ra$ and 
$\frm_4= \frm_1 \oplus \frm_2$ is the sum of the irreducible equivalent submodules 
$\frm_1  = \la e_4, e_7 \ra $ and $\frm_2 = \la e_5,e_6 \ra$.  
Consequently, the decomposition $\frm= \frm_3 \oplus \frm_4$ is orthogonal 
with respect to any $\Ad(\H)$-invariant inner product. 
Moreover, every $\Ad(\H)$-invariant inner product $g$ on $\frm$ has the following expression 
\begin{equation}\label{IPAW11}
\begin{split}
g	&= h_{\frm_3} +  a^2 \left(e^4\odot e^4 + e^7\odot e^7\right) + b^2 \left(e^5\odot e^5 + e^6\odot e^6\right) \\
	&\quad+2r\left(e^4\odot e^5 - e^6\odot e^7\right) +2s\left(e^4\odot e^6 + e^5\odot e^7\right),
\end{split}
\end{equation}
where $h_{\frm_3}$ is an arbitrary inner product on $\frm_3$, $a,b$ are positive constants and $r,s\in\R$ satisfy $r^2+s^2 < a^2b^2$. 

The space of $\Ad(\H)$-invariant $4$-forms decomposes as follows
\[
(\L^4 \frm^*)^\H \cong \L^4 \frm_4^* \oplus \L^2 \frm_3^* \otimes (\L^2 \frm_4^*)^\H, 
\]
where 
\[
(\L^2 \frm_4^*)^\H \cong \L^2 \frm_1^* \oplus \L^2 \frm_2^* \oplus (\frm_1 \otimes \frm_2)^\H,
\]
and 
\[
\L^2 \frm_1^* = \left\la e^{47}\right\ra,\quad \L^2 \frm_2^* = \left\la e^{56}\right\ra,\quad
 (\frm_1 \otimes \frm_2)^\H =  \left\la \o_2 \coloneqq  e^{46}- e^{57},\o_3 \coloneqq e^{45}+ e^{67} \right\ra.
\]
We also let $\o_0 \coloneqq e^{47}- e^{56}$ and $\o_1\coloneqq e^{47}+ e^{56}$, so that $(\L^2 \frm_4^*)^\H  = \la \o_0,\o_1,\o_2,\o_3 \ra$. 

\smallskip

As for the $7$-sphere $S^7=\Sp(2)/\Sp(1)$, we let 
$\G=\Sp(2)$ and $\H=\Sp(1)$, with the latter acting on the first summand of $\R^8=\mathbb{H}\oplus \mathbb{H}$. 
We have a reductive decomposition $\frg= \frm \oplus \frh$, with $\frm$ spanned by 
\[
e_1 = \begin{pmatrix}
0 & 0 \\
0 & i 
\end{pmatrix},\quad 
e_2 = \begin{pmatrix}
0 & 0 \\
0 & j
\end{pmatrix},\quad
e_3=\begin{pmatrix}
0 & 0 \\
0 & k
\end{pmatrix}, 
\]
\[
e_4=  \frac{1}{\sqrt{2}}\begin{pmatrix}
0 & i \\
i & 0
\end{pmatrix},\quad 
e_5 = \frac{1}{\sqrt{2}}\begin{pmatrix}
0 & j \\
j & 0
\end{pmatrix},\quad 
e_6=  \frac{1}{\sqrt{2}}\begin{pmatrix}
0 & k \\
k & 0\end{pmatrix},\quad 
e_7=  \frac{1}{\sqrt{2}}\begin{pmatrix}
0 & 1 \\
-1 & 0\end{pmatrix}, 
\]
and $\frh$ spanned by
\[
e_8 = \begin{pmatrix}
i & 0 \\
0 & 0 
\end{pmatrix}, \quad 
e_9=\begin{pmatrix}
j & 0 \\
0 & 0
\end{pmatrix}, \quad 
e_{10}=\begin{pmatrix}
k & 0 \\
0 & 0
\end{pmatrix}.
\]
 
The tangent space $\frm$ to $\G/\H$ at the identity coset decomposes into $\Ad(\H)$-invariant summands as $\frm= \frm_3 \oplus \frm_4 $, 
with $\frm_3= \la e_1, e_2, e_3 \ra$ and $\frm_4= \la e_4, e_5,e_6, e_7 \ra$. 
$\H$ acts trivially on $\frm_3$ and irreducibly on $\frm_4$, which is isomorphic to the $\H$-module $\mathbb{H}\cong\R^4$. 
Consequently, the generic $\Ad(\H)$-invariant inner product $g$ on $\frm$ has the following expression  
\begin{equation}\label{IPS7}
g	= h_{\frm_3} +  a^2 \left(e^4\odot e^4+ e^5\odot e^5 + e^6\odot e^6 + e^7\odot e^7\right),
\end{equation}
where $h_{\frm_3}$ is an arbitrary inner product on $\frm_3$ and $a$ is a positive constant. 
In particular, the splitting $\frm= \frm_3 \oplus \frm_4 $ is $g$-orthogonal.

As for the space of $\Ad(\H)$-invariant $4$-forms, we have 
\[
(\L^4 \frm^*)^\H= \L^4 \frm_4^* \oplus \L^2 \frm_3^* \otimes (\L^2 \frm_4^*)^\H.
\] 
Since $\frm_4$ and $\R^4$ are $\H$-isomorphic, we have that $(\L^2 \frm_4^*)^\H\cong \L^2_+(\R^4)^*$ and 
a basis is given by the $\Ad(\H)$-invariant $2$-forms 
$\o_1 \coloneqq e^{47}+ e^{56}$, $\o_2 \coloneqq e^{46}-e^{57}$, $\o_3\coloneqq e^{45}+ e^{67}$. 
In this case, we let $\omega_0\coloneqq 0$. 

Note that $\o_1\W\o_2 = \o_1\W\o_3=\o_2\W\o_3 = 0$ and that $\omega_k^2= 2  e^{4567}$ for $k=1,2,3$. 

\smallskip

In the following, $(\G,\H)$ denotes either $(\Sp(2),\Sp(1))$ or $(\SU(3),\U(1)_{1,1})$, 
and the tangent space to $\G/\H$ at the identity coset is identified with $\frm = \frm_3\oplus\frm_4$.  
In both cases, $\frm_3=\la e_1,e_2,e_3\ra$ 
and $\frm_4=\la e_4,e_5,e_6,e_7\ra$ are orthogonal with respect to any $\Ad(\H)$-invariant inner product on $\frm$. 

Using the notation introduced above, we have
\[
de^1= -2 e^{23}- \o_1,  \quad
de^2= 2 e^{13} + \o_2,  \quad
de^3= -2 e^{12} - \o_3,
\]
whence we deduce the following
\begin{align} \label{ec:diffs-2-1}
de^{23}&= - \frac{1}{2} d\o_1 =  e^2\W \o_3 + e^3\W \o_2 ,  \\\label{ec:diffs-2-2}
de^{13}&= - \frac{1}{2} d\o_2 =  e^1\W \o_3 - e^3\W \o_1    , \\\label{ec:diffs-2-3}
de^{12}&= - \frac{1}{2} d\o_3 =  - e^1\W\o_2 - e^2\W \o_1  .
\end{align}
Moreover $d(e^{4567}) = 0 = d\o_k \wedge \o_k$, for $k=1,2,3$, and $d\o_0=0$. 

\smallskip

If there exists an $\Ad(\H)$-invariant $\G_2T$-structure on $\frm$ with non-zero Lee form $\theta$, then $\theta\in\frm_3^*$ and 
it is not restrictive assuming that $\theta = \lambda e^1$, for some non-zero constant $\lambda$, by the next lemma. 
\begin{lemma} \label{lem:Lee-e1}
If there exists an $\Ad(\H)$-invariant $\G_2T$-structure on $\frm$ with non-zero Lee form, 
then there exists an $\Ad(\H)$-invariant $\G_2T$-structure whose Lee form is proportional to $e^1$.
\end{lemma}
\begin{proof}
Define $\K\cong\SU(2)$ as follows:
\[
\K= \left\lbrace\begin{pmatrix} 
1&0 \\ 0& h \end{pmatrix} \st h \in \Sp(1) \right\rbrace \mbox{ if } \G=\Sp(2), \quad 
\K=\left\lbrace \begin{pmatrix} 
h &0 \\ 0& 1 \end{pmatrix} \st h \in \SU(2) \right\rbrace \mbox{ if } \G=\SU(3).
\]
In both cases, the conjugation by an element of $\K$ is a Lie group isomorphism that fixes $\H$. 
Moreover, $\Ad(\K)$ acts transitively on $\frm_3$ and trivially on $\frm_4$. 
From this, the assertion follows. 
\end{proof}

We now describe the expression of the generic $\Ad(\H)$-invariant $4$-form $\rho\in(\L^4\frm^*)^\H$ such that 
$d\rho = \lambda e^1\W\rho$. 

\begin{proposition} \label{prop:integrable-Sp(2)-SU(3)}
Let $\rho\in(\L^4\frm^*)^\H$ be an  $\Ad(\H)$-invariant $4$-form such that $d\rho= \lambda e^1 \wedge \rho$, where $\lambda \neq 0$.  
Then,
\begin{equation}\label{eq:rho-int}
\rho= \l_{4567}\left(e^{4567} - e^{23} \wedge \o_1\right) + e^{13} \wedge \b_2 + e^{12} \wedge \b_3,
\end{equation}
with $ \l_{4567}\in\R$ and $\b_2,\b_3 \in \la \o_0, \o_2, \o_3 \ra$. 
Moreover, the bilinear map 
$B_\rho : \frm^* \times \frm^* \to (\L^7 \frm^*)^{\otimes 2}$ 
satisfies $B_\rho\left(e^1,e^1\right)= \l_{4567}^3\, \Omega\otimes\Omega$, $B_\rho(e^4,e^4)=B_\rho(e^5,e^5)$,  
and $B_\rho\left(e^1,e^k\right)=0$ for $k=2,3$, where $\Omega=e^{1234567}$. 
Consequently, $\l_{4567}\neq0$ whenever $\rho$ is a definite $4$-form. 
\end{proposition}
\begin{proof}
We can write $\rho= \l_{4567} e^{4567}+ e^{23} \wedge \b_1 + e^{13} \wedge \b_2 + e^{12}\wedge \b_3$, with
$\b_j= \sum_{k=0}^3 \l_j^k \o_k$, where $\l_j^k$, $\l_{4567} \in \R$. We then have
\[
e^1 \wedge \rho= \l_{4567}e^{14567} + e^{123} \wedge \left( \l_1^0\, \o_0 +  \l_1^1\,\o_1 +  \l_1^2\,\o_2 + \l_1^3\, \o_3 \right).
\]
Using \eqref{ec:diffs-2-1}-\eqref{ec:diffs-2-3}, we get 
\[
\begin{split}
d\rho &= 2(\l^2_3-\l^3_2) e^{123}\W\o_1 +2(\l^3_1-\l^1_3) e^{123}\W\o_2 + 2(\l^1_2-\l^2_1) e^{123}\W\o_3 \\
	 &\quad +2(\l^3_2-\l^2_3) e^{14567} + 2(\l^3_1-\l^1_3)e^{24567} +2(\l^2_1-\l^1_2)e^{34567}. 
\end{split}	
\]
The condition $d\rho= \lambda e^1 \wedge \rho$ is then equivalent to the following system of equations 
\[
\begin{split}
&\l\l^0_1=0,\quad \l\l^1_1 = 2(\l^2_3-\l^3_2),\quad \l\l^2_1 = 2(\l^3_1-\l^1_3),\quad \l\l^3_1 = 2(\l^1_2-\l^2_1), \\
&\l\l_{4567} = 2(\l^3_2-\l^2_3),\quad \l^3_1-\l^1_3 = 0,\quad \l^2_1-\l^1_2=0. 
\end{split}
\]
Since $\l\neq0$, this gives $\l^0_1=\l^2_1=\l^3_1=\l^1_3=\l^1_2=0$ and $\l^1_1 = \tfrac{2}{\l}(\l^2_3-\l^3_2) =  -\l_{4567}$. 
From this, the first part of the statement follows.

\smallskip

We now compute the bilinear map $B_{\rho}$ as explained in Section \ref{SecDefForm7}. 
Define $\widehat{\o}_1 \coloneqq e_{47}+ e_{56}$, $\widehat{\o}_2 \coloneqq e_{46}-e_{57}$, $\widehat{\o}_3\coloneqq e_{45}+ e_{67}\in\Lambda^2\frm$.  
Moreover, let $\widehat{\o}_0 \coloneqq e_{47}- e_{56} $ if $\G=\SU(3)$ and $\widehat{\o}_0 \coloneqq 0$ if $\G=\Sp(2)$. 
Define $\widehat{\rho}\in\Lambda^3\frm$ via the identity 
$\iota_{\widehat{\rho}}\Omega = \rho$, where $\Omega=e^{1234567}$. 
Then 
\[
\widehat{\rho}= \l_{4567} \left(e_{123} - e_1 \wedge \g_1\right) + e_2 \wedge \g_2 + e_3 \wedge \g_3,
\] 
where $\g_1=  \widehat{\o}_1$, and $\g_k\in\la \widehat{\o}_0, \widehat{\o}_2, \widehat{\o}_3\ra$, for $k=2,3$. 
We then have
\[
\iota_{e^1} \widehat{\rho} \wedge \iota_{e^1} \widehat{\rho} \wedge \widehat{\rho}  = 
\l_{4567}^2 \left(e_{23} - \g_1\right)^2\W\widehat{\rho} = 6 \l_{4567}^3\,e_{1234567},
\]
since $\g_2 \wedge \g_1=0 = \g_3\W\g_1$. Therefore $B_\rho\left(e^1,e^1\right)= \l_{4567}^3\, \Omega\otimes\Omega$. 

To prove $B_{\rho}(e^1,e^2)=0$, it is sufficient observing that 
\[
 \iota_{e^1} \widehat{\rho} \wedge \iota_{e^2} \widehat{\rho} \wedge \widehat{\rho} 
 =  \l_{4567}\left(e_{23}\wedge \g_2 -e_{13} \wedge \g_1\right)\wedge \widehat{\rho} =0. 
\]
A similar computation yields $B_{\rho}(e^1,e^3)=0$. 
Finally, the identity $B_\rho(e^4,e^4)=B_\rho(e^5,e^5)$ is always true when $\G=\Sp(2)$, since the $\H$-module $\frm_4$ is irreducible, 
while it follows from a direct computation when $\G=\SU(3)$. 
\end{proof}

\begin{remark}\label{remLeeAW11S7}
It follows from the proof of the previous result that we must have $\l^2_3\neq\l^3_2$
when the $4$-form $\rho$ is definite and $\lambda$ is not zero. 
\end{remark}

Notice that the $4$-form $\rho$ given in \eqref{eq:rho-int} depends on $7$ real parameters when $\G=\SU(3)$, while it depends on $5$ real parameters 
when $\G=\Sp(2)$, since in this last case $\o_0=0$.  

Let us now assume that $\rho$ is definite. 
This is equivalent to saying that the parameters belong to a certain open subset of $\R^{7}$ (respectively $\R^5$) 
when $\G=\SU(3)$ (respectively $\G=\Sp(2)$). This subset is non-empty, see Remark \ref{RemEG2T} below. 
Let $\overline{g}_\rho:\frm^*\times\frm^*\to\R$ be the $\Ad(\H)$-invariant metric induced by $\rho$ as explained in Section \ref{SecDefForm7}. 
Choosing an orientation on $\frm$, we have $\rho = \star\f$, 
where $\f$ is an $\Ad(\H)$-invariant $\G_2T$-structure on $\frm$ with Lee form $\theta=\l e^1$ inducing the metric $g_\f = {\overline{g}_\rho}^{-1}$.  

Choosing the orientation $\Omega= e^{1234567}$, it follows from the second part of Proposition \ref{prop:integrable-Sp(2)-SU(3)} that  
$g_\f$ is of the form
\begin{equation}\label{AwS7met}
\begin{split}
g_\f	&= \sum_{i=1}^3 a_i^2 e^i\odot e^i + 2t \, e^2\odot e^3 + 
	b^2 \left(e^4\odot e^4+ e^5\odot e^5 + e^6\odot e^6 + e^7\odot e^7\right) \\
	&\quad+2r\left(e^4\odot e^5 - e^6\odot e^7\right) +2s\left(e^4\odot e^6 + e^5\odot e^7\right),
\end{split}
\end{equation}
where the coefficients depend on those of $\rho$ and $r=s=0$ if $\G=\Sp(2)$. 
An inspection of their expressions points out the identity
\[
\left(\det\left.g_\f\right|_{\la e_2,e_3\ra}\right)^2 = \det\left.g_\f\right|_{\frm_4}. 
\]
Consequently, we have 
\[
\star\left(e^{4567}\right) =\frac{\sqrt{\det g_\f}}{ \det\left.g_\f\right|_{\frm_4}}\,e^{123},\quad 
\star\left(e^{23}\W\o_1\right) = \frac{\sqrt{\det g_\f}}{ \det\left.g_\f\right|_{\frm_4}}\,e^{1}\W\o_1, 
\]
and $\star(e^{13}\W\b_2),\star(e^{12}\W\b_3)\in\la e^2\W\o_0,e^2\W\o_2,e^2\W\o_3, e^3\W\o_0,e^3\W\o_2,e^3\W\o_3\ra$.
Collecting all these observations, we obtain the following. 
\begin{lemma}\label{lem:formula-G2}
Let $\varphi$ be an $\Ad(\H)$-invariant $\G_2T$-structure on $\frm$ with Lee form $\theta = \l e^1$, for some $\l\neq0$.  
Then 
\begin{equation} \label{eqn:G2}
\varphi= \m \left(e^{123} - e^1 \wedge \o_1\right) + e^2 \wedge \b_2' + e^3 \wedge \b_3',
\end{equation}
where $\m$ is a non-zero constant and $\b_2', \b_3' \in \la \o_0, \o_2, \o_3 \ra$. 
\end{lemma}

\begin{remark}\label{RemEG2T}\
\begin{enumerate}[i)]
\item The converse of the previous result does not hold. 
Indeed, if the $3$-form \eqref{eqn:G2} is definite, then its Lee form is given by $\theta=\l e^1$, with $\l\in\R$.
Thus, it defines either a $\G_2T$-structure with non-zero Lee form proportional to $e^1$ or a coclosed $\G_2$-structure. 
On the other hand, there exist $\Ad(\H)$-invariant coclosed $\G_2$-structures on $\frm$ that are not of the form \eqref{eqn:G2}.   
\item The space of $\Ad(\H)$-invariant $\G_2T$-structures on $\frm$ with non-zero Lee form $\theta=\l e^1$ is non-empty 
both when $\G=\SU(3)$ and when $\G=\Sp(2)$. For instance, the definite $3$-form 
\[
\f = \frac{\l}{4}\left(e^{123} - e^1 \wedge \o_1\right) + e^2 \wedge (-\o_3) + e^3 \wedge (-\o_2),
\]
defines such a type of $\G_2$-structure in both cases. The corresponding inner product is $g_\f=\frac{\l^2}{16}\,e^1\odot e^1 + \sum_{i=2}^7e^i\odot e^i$. 
\end{enumerate}
\end{remark}

Let $T$ be the torsion $3$-form of the generic $\Ad(\H)$-invariant $\G_2T$-structure $\f$ given by \eqref{eqn:G2}. 
We now focus on the condition $dT=0$ or, equivalently, $T=d\gamma$ with $\gamma \in (\L^2 \frm^*)^\H$. 
Observe that the space of $\Ad(\H)$-invariant exact $3$-forms is $d(\Lambda^2 \frm^*)^\H=\la d\o_1, d\o_2, d\o_3\ra$ 
 both when $\G=\SU(3)$ and when $\G=\Sp(2)$.

\begin{lemma}\label{lem:TorsionPerp}
Let $\f$ be an $\Ad(\H)$-invariant  $\G_2T$-structure on $\frm$ with non-zero Lee form $\theta = \lambda e^1$, and let $T$ be its torsion $3$-form. 
Then, $d\o_k\W \star T =0$, for $k=2,3$.
\end{lemma}
\begin{proof}
From equation \eqref{eqn:tor}, we obtain $\star T= \frac{7}{6}\t_0\, \star\f - d\varphi + \lambda e^1 \wedge \varphi$, 
where $\f$ is given by \eqref{eqn:G2} and $\star\f$ is of the form \eqref{eq:rho-int}.  
The thesis follows taking into account formulas \eqref{ec:diffs-2-1}-\eqref{ec:diffs-2-3}, and the equalities 
$\a \wedge e^1 \wedge \o_k =0 = \a \wedge e^k \wedge \o_1$, which hold for $\a \in \{ \star\f, d\varphi, e^1 \wedge \varphi \}$ and $k=2,3$. 
\end{proof}

\begin{proposition}\label{PropNonExAW11S7}
Let $(\G,\H)=(\SU(3),\U_{1,1})$ or $(\G,\H)=(\Sp(2),\Sp(1))$. Then, the space $\G/\H$ does not admit any $\G$-invariant strong $\G_2T$-structure.
\end{proposition}
\begin{proof} 
Assume by contradiction that there exists an $\Ad(\H)$-invariant strong $\G_2T$-structure $\f$ on $\frm$.  
Then, its Lee form $\theta$ is not zero by Proposition \ref{lem:co-closed}, and by Lemma \ref{lem:Lee-e1}
it is not restrictive assuming that $\theta =\l e^1$, for some $\l\neq0$. 
Then, $\f$ must be of the form \eqref{eqn:G2} by Lemma \ref{lem:formula-G2}. 
Since the torsion $3$-form $T$ is $\Ad(\H)$-invariant and closed, we must have $T\in d(\Lambda^2 \frm^*)^\H=\la d\o_1, d\o_2, d\o_3\ra$. 
Moreover, $T$ must be orthogonal to the subspace $\la d\o_2, d\o_3\ra$ by Lemma \ref{lem:TorsionPerp}. 
Using the expression of the inner product induced by $\f$ (cf.~\eqref{AwS7met}), we see that for $k=2,3$
\[
\star(e^k\W\o_1) \in \la e^{12}\W\o_1, e^{13}\W\o_1 \ra,\qquad \star(e^1\W\o_k) \in \la e^{23}\W\o_0, e^{23}\W\o_2, e^{23}\W\o_3 \ra. 
\]
Since $\o_1\W\o_2 = 0 = \o_1\W\o_3$, from the identities \eqref{ec:diffs-2-1}-\eqref{ec:diffs-2-3} we see that $d\o_1$ is orthogonal to the 
subspace $\la d\o_2, d\o_3\ra$, too. Therefore, $T$ must be proportional to $d\o_1 = -2\left(e^2\W\omega_3 + e^3\W\omega_2  \right)$, 
and thus $\star T \in\la e^{12}\W\o_0, e^{12}\W\o_2, e^{12}\W\o_3, e^{13}\W\o_0,  e^{13}\W\o_2, e^{13}\W\o_3 \ra$. 
In particular, we must have  $\star T \wedge e^{23}=0$ and  $\star T \wedge \o_1=0$. 

Now, $\star T= \frac{7}{6}\t_0\, \star\f - d\varphi + \lambda e^1 \wedge \varphi$, 
where $\star\f$ is of the form \eqref{eq:rho-int} with $\beta_k = \l^0_k\o_0+\l^2_k\o_2+\l^3_k\o_3$, 
and $\f$ is of the form \eqref{eqn:G2} with $\beta_k' = \mu^0_k\o_0+\mu^2_k\o_2+\mu^3_k\o_3$, for $k=2,3$. 
Using \eqref{ec:diffs-2-1}-\eqref{ec:diffs-2-3}, we obtain 
\[
\begin{split}
d\f\W e^{23} &= d(\f\W e^{23}) + \f\W d(e^{23}) = \left(2\mu+2\mu^2_2-2\mu^3_3\right) e^{234567},\\
d\f\W \o_1 &= d(\f\W \o_1) + \f \W d\o_1 = \left(2\mu-4\mu^2_2+4\mu^3_3\right) e^{234567}. 
\end{split}
\] 
From the expression of $\star\f$ we deduce that
\[
\star\f\W e^{23} = \l_{4567}\,e^{234567} = -\frac12 \star\f\W\o_1.
\]
Finally, we see that $e^1\W\f\W e^{23} = 0 =e^1\W\f\W\o_1$.
We then have 
\[
\begin{split}
\star T \wedge e^{23} =& \left( \frac{7}{6} \t_0 \l_{4567} - 2 \mu - 2  \m_2^2 + 2  \m_3^3\right) e^{234567}, \\
\star T \wedge \o_1=& -2 \left( \frac{7}{6} \t_0 \l_{4567} + \m - 2 \m_2^2 + 2 \m_3^3 \right) e^{234567}.
\end{split}
\]
Therefore $\frac{7}{6} \t_0 \l_{4567} - 2 \mu - 2  \m_2^2 + 2  \m_3^3 = 0 = \frac{7}{6} \t_0 \l_{4567} + \m - 2 \m_2^2 + 2 \m_3^3$. 
This implies $\m=0$, which is a contradiction. 
\end{proof}

\subsubsection*{\underline{The cases $2$, $3$ and $5$, $6$, $7$ of Table \ref{TabIrred}}}\  \vspace{0.2cm}  

\noindent 
Before showing that these spaces do not admit invariant strong $\G_2T$-structures, we briefly describe their homogeneous structure. 

Case $2$ corresponds to the orbit of a transitive action of $\SU(3)\times \U(1)$ on $N^{p,q}$. 
The group $\SU(3)$ acts by left multiplication on $N^{p,q}$, while $\U(1)$ acts by right multiplication and it is identified with 
a $1$-dimensional subgroup of the normalizer $N_{\SU(3)}{\U(1)_{p,q}}$. In sum, the action is
$
(s,u)\cdot x \U(1)_{p,q}= s x u^{-1}\U(1)_{p,q}.
$
Its isotropy is isomorphic to $\U(1)^2$ and there is a covering ${\pi}:N^{p,q} \to \SU(3) \times \U(1) /\U(1)^2$.

Case $3$ is the orbit of a transitive action of $\SU(3)\times \SU(2)$ on $N^{1,1}$. 
The group $\SU(3)$ acts by left multiplication. 
The group $\SU(2)$ is identified with $\mathrm{S}(\U(2) \times \U(1))\subset \SU(3)$ and it acts by right multiplication. 
The action is thus 
$
(s,u)\cdot x \U(1)_{1,1}= sx u^{-1}\U(1)_{1,1}.
$
Its isotropy is $\SU(2)\times \U(1)$ and we have a diffeomorphism $N^{1,1}\cong \SU(3)\times \SU(2)/\SU(2)\times \U(1)$.

Cases $4$-$7$ are obtained from transitive actions of different groups on the $7$-sphere $S^7$. 
Cases $4$ and $5$ correspond to the action by left multiplication of the groups $\Sp(2)$ and $\SU(4)$, 
with isotropy $\Sp(1)$ and $\SU(3)$, respectively. 
In cases $6$ and $7$ we have the left multiplication by $\Sp(2)$ and the right multiplication by $\U(1)$ or $\Sp(1)$. 
Here, we view $\Sp(1)$ as the subspace of unit quaternions acting on $\R^8=\mathbb{H} \oplus \mathbb{H}$ and $\U(1)\subset \Sp(1)$.
Again, the actions are
$(s,u) \cdot x= sxu^{-1}$, with $s \in \Sp(2)$, {$x\in S^7$},  and $u\in \Sp(1)$ or $u \in \U(1)$.
As the kernel of both isotropy representations is {isomorphic to $\Z_2$}, there are $2:1$ 
coverings ${\pi:S^7=\Sp(2)/\Sp(1)} \to \Sp(2)\times \Sp(1)/\Sp(1)^2$ and ${\pi: S^7=\Sp(2)/\Sp(1)} \to \Sp(2)\times \U(1)/\Sp(1) \times \U(1)$.

\begin{proposition}
None of the irreducible homogeneous spaces listed in cases $2,3$ and $5,6,7$ of Table \ref{TabIrred} admits invariant strong $\G_2T$-structures.
\end{proposition}
\begin{proof}
In cases $2$ and $3$, the map $\pi \colon N^{p,q} \to \G/\H$ induced by the inclusion $\SU(3) \subset \G$ yields a local isomorphism. 
In cases $5$-$7$, the same holds true for the map $ \pi \colon \Sp(2)/\Sp(1) \to \G/\H$. 
Therefore, if there were an invariant strong $\G_2T$-structure $\varphi$ on $\G/\H$, 
then $\pi^* \varphi$ would be an invariant strong $\G_2T$-structure on $N^{p,q}$ or $\Sp(2)/\Sp(1)$. 
{This is not possible by Proposition \ref{PropNonExAW11S7}.  }
\end{proof}

\subsection*{Cases 8-12 of Table \ref{TabIrred}} \label{subsec:rest}

We start with the $7$-sphere $S^7=\Spin(7)/\G_2$ and the Berger space $B^7=\SO(5)/\SO(3)$, which is a rational homology sphere. 
The description of the embedding $\SO(3)\subset \SO(5)$ can be found in \cite{BerH}. 
{These spaces are isotropy irreducible, thus they 
do not admit any invariant $1$-form and Proposition \ref{lem:co-closed} immediately gives the following. }
\begin{proposition}
There are no invariant strong $\G_2T$-structures on $S^7=\Spin(7)/\G_2$ and on $B^7=\SO(5)/\SO(3)$.
\end{proposition}

We now describe the structure of the remaining homogeneous spaces $V^{5,2}$, $Q^{1,1,1}$ and $M^{1,1,0}$. 
The first space is $V^{5,2}=\SO(5)/\SO(3)$, where $\SO(3) \subset \SO(5)$ is the standard embedding. 
The second one is $Q^{1,1,1}=\SU(2)^3/\U(1)^2$, where $\U(1)^2$ denotes the subgroup
\[ 
\left\{ 
\diag\left(zw,(zw)^{-1},zw,(zw)^{-1},(zw)^{-1},zw\right)\st (z,w)\in\U(1)^2
\right\} 
\subset \SU(2)^3.
\]
Finally, $M^{1,1,0}= \SU(3)\times \SU(2)/\SU(2) \times \U(1)$, 
where $\SU(2)$ is viewed as a subgroup of $\SU(3)$ and $\U(1)$ is transversely embedded as 
{$\left\{\diag\left(z,z,z^{-2},z^3,z^{-3}\right) \st z \in \U(1) \right\}$. }

{To show that none of these spaces admits invariant strong $\G_2T$-structures, we use the following. }

\begin{lemma}\label{lem:2-forms-10-12}
Every invariant $2$-form on $V^{5,2}$, $Q^{1,1,1}$, and $M^{1,1,0}$ is closed. 
\end{lemma}
\begin{proof}
For $V^{5,2}$, let $\G=\SO(5)$ and $\H=\SO(3)$. 
Let us consider the basis $\{E_{ij}\}_{1\leq i<j\leq 5}$ of $\frg= \mathfrak{so}(5)$, {where $E_{ij}$ denotes the skew-symmetric $5\times5$ matrix 
whose only non-zero entries are $1$ in position $(i,j)$ and $-1$ in position $(j,i)$. }
We have a reductive decomposition $\frg =  \frm \oplus \frh$, where $\frh = \la E_{12},E_{13},E_{23} \ra$. 
The action of $\H$ yields a decomposition $\frm= \frm_0 \oplus \frm_1 \oplus \frm_2$, where the invariant irreducible summands are 
$\frm_0=\la E_{45} \ra$, $\frm_1=\la E_{14}, E_{24}, E_{34} \ra$, and $\frm_2=\la E_{15}, E_{25}, E_{35} \ra$. 
The $\H$-module $\frm_0$ is trivial, and the $\H$-modules $\frm_1$ and $\frm_2$ are isomorphic. 
We define the basis vectors 
$e_i \coloneqq E_{i4}$, $e_{i+3} \coloneqq E_{i5}$, for $i\leq 3$, and $e_7 \coloneqq E_{45}$, and we denote the dual basis by 
$(e^1, \dots , e^7)$.
Notice that $\H$ acts on $\frm_j$ as $\SO(3)$ for $j=1,2$. This implies that $(\L^2 \frm^*)^\H= (\frm_1^*\otimes \frm_2^*)^\H$. We have  
$(\frm_1^*\otimes \frm_2^*)^\H=\la e^{14}+ e^{25} + e^{36} \ra = \la de^7 \ra$, where for the last equality we used that $[E_{i4}, E_{i5}]= E_{45}$. 
Therefore, $d(\Lambda^2 \frm^*)^\H=0$.

\smallskip

For $Q^{1,1,1}$ we consider $\G=\SU(2)^3$ and $\H=\U(1)^2$.
{Then, $\frg = \frsu(2) \oplus \frsu(2)\oplus \frsu(2)$ and }
we have a reductive splitting $\frg=\frm \oplus \frh$ with $\frm$ spanned by
\[
\begin{split}
& e_1 = (\sigma_2,0,0),\quad e_2 =  (\sigma_3,0,0), \quad e_3= (0,\sigma_2,0),\quad e_4= (0 ,\sigma_3,0),\\
& e_5= (0,0,\sigma_2),\quad e_6=(0,0,\sigma_3),\quad e_7=(\sigma_1,\sigma_1,\sigma_1),
\end{split}
\]
and $\frh = \la e_8= (\sigma_1,0,-\sigma_1),~e_9=(0,\sigma_1,-\sigma_1) \ra $, where the matrices
\begin{equation}\label{eq:sigma-generators}
\sigma_1=\begin{pmatrix}
i & 0 \\
0 & -i
\end{pmatrix}, \quad   
\sigma_2= \begin{pmatrix}
0 & 1 \\
-1 & 0 
\end{pmatrix},\quad 
\sigma_3= \begin{pmatrix}
0 & i \\
i & 0 
\end{pmatrix},
\end{equation}
define a basis of $\frsu(2)$. 
The decomposition of $\frm$ into $\Ad(\H)$-invariant irreducible subspaces is 
$\frm= \frm_1 \oplus \frm_2 \oplus \frm_3 \oplus \frm_4$, with $\frm_j= \la e_{2j-1},e_{2j} \ra$ for $j\leq 3$ and $\frm_4=\la e_7 \ra$. 
In addition, for $j\leq3$, $\H$ acts on $\frm_j$ as $\SO(2)$. 
We have that
$(\Lambda^2 \frm^*)^\H = \oplus_{j=1}^3 \Lambda^2 \frm_j^*$. Therefore, $d(\Lambda^2 \frm^*)^\H=0$.

\smallskip

For $M^{1,1,0}$ we consider $\G=\SU(3)\times \SU(2)$ and $\H=\SU(2) \times \U(1)$. 
Let $(e_1,\dots, e_8)$ be the basis of $\mathfrak{su}(3)$ {defined in the proof of Proposition \ref{PropAWpq} with $p=q=1$,} 
and let $(\sigma_1,\sigma_2,\sigma_3)$ be the basis of $\mathfrak{su}(2)$ defined above.  
{The Lie algebra $\frg = \frsu(3)\oplus\frsu(2)$ admits} a reductive splitting $\frg=\frm \oplus \frh$, 
where $\frh = \la f_8={(e_1,0)},
f_9=(e_2,0), f_{10}=(e_3,0), f_{11}= (e_8, 3\, \sigma_1)\ra$
and $\frm$ is spanned by $f_i=(e_{i+3},0)$, $i\leq 4$, $f_5=(0,\sigma_2)$, $f_6=(0,\sigma_3)$, and $f_7=({\diag(i,i,-2i)},0)$.
We denote by $(f^1,\dots, f^7)$ the dual basis. 
The decomposition of $\frm$ into $\Ad(\H)$-invariant irreducible summands is $\frm=\frm_1 \oplus \frm_2 \oplus \frm_3$, 
where $\frm_1=\la f_1,f_2,f_3,f_4\ra$, $\frm_2=\la f_5,f_6\ra$ and $\frm_3=\la f_7 \ra$. 
In addition, {$df^7=-\tfrac23 (f^{12} + f^{34} + f^{56})$. } 
We have that $(\Lambda^2 \frm^*)^\H= (\Lambda^2 \frm_1^*)^\H \oplus\Lambda^2 \frm_2^*=\la f^{12}+ f^{34},f^{56}\ra$. 
Thus, $(\Lambda^2 \frm^*)^\H=\la df^7, f^{56} \ra$. Since $d(f^{56})=0$, we obtain $d(\Lambda^2 \frm^*)^\H=0$.
\end{proof}

\begin{proposition} \label{prop:10-12}
There are no invariant strong $\G_2T$-structures on $V^{5,2}$, $Q^{1,1,1}$, and $M^{1,1,0}$. 
\end{proposition}
\begin{proof}
Let $\G/\H$ be one of these spaces. 
Then, $b_3(\G/\H)=0$ and $d(\Lambda^2 \frm^*)^\H=0$ by Lemma \ref{lem:2-forms-10-12}.
Therefore, the map $d \colon (\Lambda^3 \frm^*)^\H \to (\Lambda^4 \frm^*)^\H$ is an isomorphism.  
This implies that  every $\Ad(\H)$-invariant $4$-form on $\frm$ is closed. 
Proposition \ref{lem:co-closed} allows us to conclude that there are no invariant strong $\G_2T$-structures on $\G/\H$.
\end{proof}

%
\section{Reducible compact  homogeneous spaces}\label{sec:Reducible}

By \cite{Reid}, a reducible compact homogeneous space admitting invariant $\G_2$-structures must be one of   
those listed in Table \ref{TabRed}, up to a covering.    
We will make a case-by-case analysis to search for invariant strong $\G_2T$-structures on such spaces.

\begin{table}[h!]
\centering
\renewcommand\arraystretch{1.2}
\begin{tabular}{|c|c|c|c|}  
\hline
& $\G$& $\H$& $M=\G/\H$ \\ \hline 
1&	$\U(1)^7$				& $\{1_\G\}$ 	& $T^7$ \\ \hline
2&	$\SU(2)\times \U(1)^4$	& $\{1_\G\}$	& $S^3\times T^4$ \\ \hline
3&	$\SU(2)^2\times \U(1)$& $\{1_\G\}$ & $S^3\times S^3\times S^1$ \\ \hline
4&	$\SU(2)^2\times \U(1)^2$& $\U(1)$ & $S^3\times S^3\times S^1$ \\ \hline
5&	$\SU(2)^3\times \U(1)$& $\SU(2)$ & $S^3\times S^3\times S^1$ \\ \hline
6&	$\SU(3)\times \U(1)^2$& $\SU(2)$ & $S^5\times T^2$ \\ \hline
7&	$\SO(4)\times \U(1)^2$& $\SO(2)$ & $V^{4,2}\times T^2$ \\ \hline
8&	$\SU(3)\times \U(1)$& $\U(1)^2$ & ${F}(1,2)\times S^1$ \\ \hline
9&	$\Sp(2)\times\U(1)$& $\Sp(1)\times\U(1)$ & $\mathbb{C}P^3\times S^1$ \\ \hline
10&	$\G_2\times\U(1)$& $\SU(3)$ & $S^6\times S^1$ \\ \hline
\end{tabular}  \vspace{0.2cm}
\caption{Reducible compact homogeneous spaces admitting invariant $\G_2$-structures, up to a covering \cite{Reid}. }\label{TabRed}
\end{table}

We first deal with the case of $7$-dimensional reducible compact Lie groups admitting left-invariant $\G_2$-structures, namely  
$\U(1)^7\cong T^7$, $\SU(2)\times \U(1)^4\cong S^3\times T^4$ and $\SU(2)^2\times \U(1)\cong S^3 \times S^3 \times S^1$. 
{Recall that a left-invariant $\G_2$-structure $\f$ on a Lie group $\G$ is determined by the definite $3$-form $\f_{\sst{1_\G}}$ 
on $T_{\sst{1_\G}}\G\cong\frg$ and, conversely, left multiplication extends any definite $3$-form $\f$ on a Lie algebra $\frg$ 
to a left-invariant $\G_2$-structure on every corresponding Lie group. Thus, we shall work on the Lie algebra $\frg$ of $\G$. }

Left-invariant $\G_2$-structures on $T^7$ are clearly parallel. We now show the existence of left-invariant non-parallel 
strong $\G_2T$-structures on the remaining Lie groups.

\begin{proposition}\label{S3T4G2T}
The Lie group $\G=\SU(2)\times \U(1)^4$ admits left-invariant strong $\G_2T$-structures solving the twisted $\G_2$ equation \eqref{eqn:strom7}. 
\end{proposition}
\begin{proof}
The homogeneous Hopf surface $S^3 \times S^1 \cong \SU(2)\times\U(1)$ admits a left-invariant $\SU(2)$-structure $(\o,\Psi)$ 
solving the twisted Calabi-Yau equation \eqref{TCY} (cf.~\cite{GRST}). 
We can then obtain a left-invariant strong $\G_2T$-structure on $\G \cong  \SU(2)\times\U(1) \times \U(1)^3$ 
solving the twisted $\G_2$ equation \eqref{eqn:strom7} as explained in the proof of Proposition \ref{TCYG2link}.

Let us consider the splitting $\frg= \frg_3\oplus\frg_4$, with $\frg_3=\R^3$ and $\frg_4=\mathfrak{su}(2)\oplus \R$. 
We choose a basis $(e_1,\ldots,e_7)$ of $\frg$ such that $\frg_3=\la e_1,e_2,e_3\ra$, $\frsu(2) = \la e_4,e_5,e_6\ra$, $\R=\la e_7\ra$, 
and whose dual basis $(e^1,\ldots,e^7)$ satisfies the structure equations 
\[
de^1=de^2=de^3=de^7=0,\quad de^4=e^{56},\quad de^5=-e^{46},\quad de^6=e^{45}. 
\] 
Consider the $\SU(2)$-structure on $\frg_4$ defined by the forms 
$\o = e^{45}+e^{67}$, $\psip=e^{46}-e^{57}$, $\psim=e^{47}+e^{56}$ and the inner product $g = \sum_{i=4}^7e^i\odot e^i$. 
Its Lee form is $\theta_\o = e^7$ and we easily see that $(\o,\Psi)$ satisfies the equations \eqref{TCY}. 
Notice that  $d^c\omega = Jd\omega = e^{456}$.
Therefore, the $3$-form 
\[
\varphi=e^{123}+ e^1\wedge \o + e^2\wedge \psip - e^3\wedge \psim,
\]
defines a strong $\G_2T$-structure on $\frg$ with Lee form $\theta=e^7$ and torsion $3$-form $T=d^c\o$ inducing the inner product 
$g_\f = \sum_{i=1}^3 e^i\odot e^i + g$ and solving the twisted $\G_2$ equation \eqref{eqn:strom7}. 

\smallskip 

The metric $g$ gives rise to a bi-invariant metric on $\SU(2)$ of constant sectional curvature $\tfrac14$ and  
the torsion $3$-form of the strong $\G_2T$-structure corresponds to the bi-invariant harmonic $3$-form $T = -g([\cdot,\cdot],\cdot)$ of $\SU(2)$. 
\end{proof}

\begin{proposition}\label{S3S3S1G2T}
The Lie group $\G=\SU(2)^2\times \U(1)$ admits left-invariant strong $\G_2T$-structures 
solving the twisted $\G_2$ equation \eqref{eqn:strom7} and coclosed $\G_2$-structures with closed torsion $3$-form. 
\end{proposition}
\begin{proof}

We describe an $\SU(3)$-structure on $\frsu(2)\oplus\frsu(2)$ that satisfies the hypothesis of (\ref{prodclosed1}) 
of Proposition \ref{prop:product-closedT} with $w_2^+=0$ and thus gives rise to a strong $\G_2T$-structure on 
$\frg = \frsu(2)\oplus\frsu(2)\oplus\R$. 
Proposition \ref{prop:product-str} guarantees then the existence of a strong $\G_2T$-structure solving the 
twisted $\G_2$ equation \eqref{eqn:strom7} and of a coclosed $\G_2$-structure on $\frg$. 
All these structures share the same metric and harmonic torsion $3$-form $T$. 

Let us choose a basis $(e^1,\ldots,e^7)$ of $\frg^*$ so that 
$\frsu(2)^*=\la e^1,e^2,e^3\ra$, $\frsu(2)^*=\la e^4,e^5,e^6\ra$, $\R^* = \la e^7\ra$, and the structure equations are the following  
\[
de^1=e^{23},~de^2=-e^{13},~de^3=e^{12},~de^4=e^{56},~de^5=-e^{46},~de^6=e^{45},~de^7=0. 
\]
Consider the $\SU(3)$-structure on $\frsu(2)\oplus\frsu(2)$ defined by  
\[
\omega = e^{14} + e^{25} - e^{36},\quad \psip = e^{123} + e^{156}-e^{246} - e^{345},	\quad \psim = e^{456} + e^{234} - e^{135}- e^{126}, 
\]
and with corresponding inner product  $g = \sum_{i=1}^6 e^i \odot e^i$. 
The only non-vanishing intrinsic torsion forms are $w_1^{\pm}=\frac{1}{2}$ and the closed $3$-form
\[
w_3 = \frac14\left(3\,e^{123}-3\,e^{456}  -e^{126} -e^{135}-e^{156}  +e^{234} +e^{246}  +e^{345}\right),
\]
so that
\[
d\o = -\frac34\psip+\frac34\psim +  w_3,\quad d\psi_{\pm} = \frac12 \o^2. 
\]
In addition, $d\star_6 w_3= -\frac{1}{4}\omega^2=-\frac12 d\left( w_1^+\psip +  w_1^- \psim\right)$.
Then, the $3$-form $\f = \o\W e^7 +\psip$ defines a strong $\G_2T$-structure inducing the inner product $g_\f = g + e^7\odot e^7$ on $\frg$. 
Its Lee form is the closed $1$-form $\theta= e^7$ and its torsion form is the harmonic $3$-form
\[
T = \frac12 w_1^+\psip + \frac12 w_1^- \psim +\star_6 w_3  = e^{123}+e^{456}. 
\]
 By Proposition \ref{prop:product-str}, the $\SU(3)$-structure $\left(\omega, \widetilde{\Psi}\right)$ with $\widetilde{\psi}_+=\frac{1}{\sqrt{2}}(\psi_+ - \psi_-)$  
gives rise to a strong $\G_2T$-structure on $\frg$ solving the twisted $\G_2$ equation \eqref{eqn:strom7}.  
On the other hand, the $\SU(3)$-structure $\left(\omega, \widehat{\Psi}\right)$ with $\widehat{\psi}_+=\frac{1}{\sqrt{2}}(\psi_+ + \psi_-)$ 
gives rise to a coclosed $\G_2$-structure on $\frg$. 
Both these $\G_2$-structures induce the inner product $g_\f$ and have torsion $3$-form $T=e^{123}+e^{456}$. 

\smallskip

In this example, the inner product $g$ gives rise to a bi-invariant metric on $\SU(2)\times \SU(2)$ that is given by the product of 
bi-invariant metrics of constant sectional curvature $\tfrac14$ on each factor. 
Moreover, the torsion $3$-form of the strong $\G_2T$-structure 
corresponds to the sum of the bi-invariant harmonic $3$-forms $T = -g([\cdot,\cdot],\cdot)$ of each factor $\SU(2)$.  
\end{proof}

\begin{remark}
The $\SU(3)$-structure $\left(\omega, \frac{1}{\sqrt{2}}(\psip - \psim)\right)$ on $\frsu(2)\oplus\frsu(2)$ is double-half flat. 
This example was obtained in  \cite{ScSc}.
\end{remark}

\begin{remark}
In the examples described in the proofs of Proposition \ref{S3T4G2T} and Proposition \ref{S3S3S1G2T}, 
the $\G_2$-connection with totally skew-symmetric torsion $\nabla = \nabla^{g_\f} +\tfrac12g_\f^{-1}T$ is flat. 
To see this, recall that every compact (semi)simple Lie group endowed with a bi-invariant metric $g$ 
admits two flat metric connections with totally skew-symmetric torsion $T(X,Y,Z)=\pm g([X,Y],Z)$ (see, e.g., \cite{AF}). 
The metric on $\SU(2)\times\U(1)^4$ is the product of a bi-invariant metric $g$ on the compact simple Lie group $\SU(2)$ and 
the flat metric on the Abelian factor. Moreover, the torsion $3$-form is the bi-invariant harmonic $3$-form $T = - g([\cdot,\cdot],\cdot)$ on $\SU(2)$. 
Similarly, the metric on $\SU(2)\times\SU(2)\times\U(1)$ is the product of two copies of the same bi-invariant metric $g$ on $\SU(2)$ and 
the flat metric on the Abelian factor, and the torsion $3$-form is the sum of the harmonic $3$-forms $T = - g([\cdot,\cdot],\cdot)$ of the simple factors. 
\end{remark}

Cases $4$ and $5$ of Table \ref{TabRed} are both homogeneous spaces diffeomorphic to $S^3 \times S^3 \times S^1$. 
In the first case, $\G= \SU(2)^2 \times \U(1)^2$ and $\H \cong \U(1)$  
is transversely embedded into ${\G}' \coloneqq \SU(2)^2 \times \U(1)$ 
as follows:
\[
\H = \left\{ \left(\diag\left(z, z^{-1}\right),\diag\left(z, z^{-1}\right), z\right) \st z \in \U(1) \right\}. 
\]
We then have a diffeomorphism
$
F \colon \SU(2) \times \SU(2) \to {\G}'/\H,
$ 
$
(g_1,g_2) \longmapsto [(g_1,g_2,1)].
$

In the second case, $\G= \SU(2)^3\times \U(1)$ and $\H\cong \SU(2)$ 
is diagonally embedded into ${\G}' \coloneqq \SU(2)^3$ as the subgroup
\[
\H=\Delta\SU(2) = \left\{ (h,h,h) \st h \in \SU(2) \right\}. 
\]
Thus, we have a diffeomorphism
$
F \colon \SU(2) \times \SU(2) \to {\G}'/\H,
$
$
(g_1,g_2) \longmapsto [(g_1,g_2,1)].
$

In both cases, it is easy to check that there is a $\G'$-invariant SU(3)-structure on $\G'/\H$ whose pull-back via $F$ is the left-invariant SU(3)-structure on 
$\SU(2)\times\SU(2)$ considered in the proof of Proposition \ref{S3S3S1G2T}. We then have the following. 

\begin{proposition}
Let $(\G,\H)=(\SU(2)^2 \times \U(1)^2,\U(1))$ or $(\G,\H)=(\SU(2)^3 \times \U(1),\SU(2))$. Then,
$\G/\H$ admits left-invariant strong $\G_2T$-structures 
solving the twisted $\G_2$ equation \eqref{eqn:strom7} and coclosed $\G_2$-structures with closed torsion $3$-form. 
\end{proposition}

We now discard the existence of invariant strong $\G_2T$-structures on the spaces $6$ and $7$ of Table \ref{TabRed}, namely 
$S^5\times T^2$ and $V^{4,2}\times T^2$. 
These spaces  share some similarities and we describe them using a unified notation that allows us arguing simultaneously 
in the first part of the discussion. However, we will need to analyze them separately at the end.

\smallskip

Consider the space $S^5\times T^2 = \SU(3)\times \U(1)^2/\SU(2)$, 
where the embedding $\SU(2) \subset \SU(3)$ is the standard one. 
Let $\G=\SU(3)\times \U(1)^2$ and $\H=\SU(2)$.  
We choose a reductive decomposition $\frg= \frm \oplus \frh$  
with $\frh=\frsu(2)$ and $\frm= \frm' \oplus \R \oplus \R$, 
where $\frm' \subset \frsu(3)$ is spanned by
\[
\begin{split}
e_1 &= \begin{pmatrix}
0 & 0 & 1 \\
0 & 0 & 0 \\
-1 & 0 & 0 
\end{pmatrix}, \quad
e_2=  \begin{pmatrix}
0 & 0 & i \\
0 & 0 & 0 \\
i & 0 & 0  \end{pmatrix},\quad
e_3=\begin{pmatrix}
0 & 0 & 0 \\
0 & 0 & -1 \\
0 & 1 & 0 
\end{pmatrix},\quad
e_4 = \begin{pmatrix}
0 & 0 & 0 \\
0 & 0 & -i \\
0 & -i & 0 
\end{pmatrix},\\
e_5 &= \diag\left(i,i,-2i\right). 
\end{split}
\]
The $\Ad(\H)$-invariant subspace $\frm'$ is identified with the tangent space to $S^5=\SU(3)/\SU(2)$ at the identity coset, and 
it decomposes into $\Ad(\H)$-invariant irreducible summands as $\frm'\cong \frm_4\oplus \la e_5\ra$, where $\frm_4=\la e_1,e_2,e_3,e_4 \ra$. 

Denote by $(e_6,e_7)$ the standard basis of the subspace $\R \oplus \R\subset\frm$, so that $\frm=\la e_1,\ldots,e_7\ra$, 
and let $(e^1,\dots,e^7)$ be the dual basis of $\frm^*$. 

There is an $\Ad(\H)$-invariant decomposition $\frm= \frm_4 \oplus \frm_3$, where $\frm_4$ is irreducible, and
$\H$ acts trivially on 
$\frm_3= \la e_5, e_6, e_7 \ra$.
Consequently, the decomposition $\frm= \frm_4 \oplus \frm_3$ is orthogonal 
with respect to any $\Ad(\H)$-invariant inner product. 
The space of $\Ad(\H)$-invariant $2$-forms decomposes as follows
\[
(\L^2 \frm^*)^\H \cong ( \L^2 \frm_4^*)^\H \oplus \L^2 \frm_3^*, 
\] 
where 
\[
(\L^2 \frm_4^*)^\H  = \left\la 
\omega_1\coloneqq e^{12}- e^{34},\, 
\omega_2\coloneqq {e^{13}+ e^{24}},\, 
\omega_3\coloneqq  e^{14}- e^{23}
 \right\ra.
\]
Moreover, we have the following equalities
\begin{equation} \label{eq:str-diff-1}
\omega_1 = -de^5, \quad
d\omega_2 = -c\, \omega_3 \wedge e^5 \quad
d\omega_3 = c \,\omega_2 \wedge e^5,
\end{equation}
where $c=6$.
The space of $\Ad(\H)$-invariant $4$-forms decomposes as 
$
(\L^4 \frm^*)^\H= \L^4 \frm_4^* \oplus \L^2 \frm_3^* \otimes (\L^2 \frm_4^*)^\H. 
$

\smallskip

Consider now the space $V^{4,2}\times T^2 = \SU(2)^2\times \U(1)^2/\U(1)$. 
Here $\G=\SU(2)^2\times \U(1)^2$ and $\H\cong \U(1)$ is diagonally embedded into $\SU(2)^2$ as the subgroup
\begin{equation}\label{U1diag}
\H = \left\{ \left(\diag\left(z,z^{-1}\right),\diag\left(z,z^{-1}\right)\right) \st z \in \U(1) \right\}.
\end{equation}
We choose a reductive decomposition $\frg= \frm \oplus \frh$ with $\frm= \frm' \oplus {\R \oplus \R}$.  
Here, $\frm' \subset \frsu(2)\oplus \frsu(2)$ is identified with the tangent space to $V^{4,2} =  \SU(2)^2 / \U(1)$ at the identity coset. 
Let $(\sigma_1,\sigma_2,\sigma_3)$ be the generators of $\frsu(2)$ described in \eqref{eq:sigma-generators}. 
Then, a basis of $\frm'$ is given by
\[
e_1 =\left(\sigma_2,0\right), \quad e_2 =\left(\sigma_3,0\right),\quad
e_3=\left(0,\sigma_2 \right), \quad e_4=\left(0,\sigma_3 \right), \quad e_5= \left(\sigma_1,-\sigma_1 \right), 
\]
and we choose the generator $\left(\sigma_1,\sigma_1 \right) $ of $\frh \subset \frsu(2)\oplus \frsu(2)$. 
Then, $\frm'$ decomposes into $\Ad(\H)$-invariant irreducible summands as 
$\frm'\cong \frm_{2,1} \oplus \frm_{2,2} \oplus \la e_5\ra$, where $\frm_{2,k}=\la e_{2k-1}, e_{2k} \ra$, for $k=1,2$. 

As in the previous case, we denote by $(e_6,e_7)$ the standard basis of the subspace $\R \oplus \R\subset\frm$, 
so that $(e_1,\dots,e_7)$ is a basis of $\frm$, and we denote by $(e^1,\dots,e^7)$ the dual basis of $\frm^*$.  

We have an $\Ad(\H)$-invariant decomposition $\frm= \frm_4 \oplus \frm_3$, where $\frm_4=\frm_{2,1} \oplus \frm_{2,2}$ and 
$\frm_3= \la e_5, e_6, e_7 \ra$ is a trivial $\Ad(\H)$-module.  
In particular, the decomposition $\frm= \frm_4 \oplus \frm_3$ is orthogonal with respect to any $\Ad(\H)$-invariant inner product on $\frm$.

The space of $\Ad(\H)$-invariant $2$-forms decomposes as follows
\[
(\L^2 \frm^*)^\H \cong \L^2 \frm_{2,1}^* \oplus \L^2 \frm_{2,2}^* \oplus ( \frm_{2,1}^*\otimes \frm_{2,2}^*)^\H \oplus \L^2 \frm_3^*, 
\] 
where 
\[
( \frm_{2,1}^*\otimes \frm_{2,2}^*)^\H  = \left\la 
\omega_2\coloneqq  e^{13}+ e^{24},\, 
\omega_3\coloneqq e^{14}- e^{23}
 \right\ra.
\]
We also let $\omega_0\coloneqq e^{12}+e^{34}$ and $\omega_1 \coloneqq e^{12}- e^{34}$. 
We have the following equalities 
\begin{equation} \label{eq:str-diff-2}
d\omega_0=0, \quad
\omega_1 = -de^5, \quad
d\omega_2 = -c\, \omega_3 \wedge e^5, \quad
d\omega_3 = c \,\omega_2 \wedge e^5,
\end{equation}
where $c=4$. 
The space of $\Ad(\H)$-invariant $4$-forms decomposes as 
$
(\L^4 \frm^*)^\H= \L^4 \frm_4^* \oplus \L^2 \frm_3^* \otimes (\L^2 \frm_4^*)^\H. 
$

\smallskip

The next lemma shows that the Lee form of any $\Ad(\H)$-invariant $\G_2T$-structure $\f$ on $\frm$ is closed in both cases under exam.  
This implies in particular that $\theta \in \la e^6, e^7 \ra$. 
\begin{lemma} \label{lem:lee-closed}
The Lee form of any $\Ad(\H)$-invariant $\G_2T$-structure $\f$ on $\frm$ is closed.
\end{lemma}
\begin{proof}
The condition $d\star \varphi=\theta \wedge \star \varphi$ yields $d\theta \wedge \star \varphi =0$. 
The $1$-form $\theta$ is $\Ad(\H)$-invariant, thus $\theta \in \frm_3^* = \langle e^5,e^6,e^7\rangle$.
Due to the orthogonal splitting $\frm= \frm_4 \oplus \frm_3$, we can write
\[
\star \varphi= \frac{1}{2}\g_1^2 + v^{23}\wedge \g_1 - v^{13} \wedge \g_2 - {v^{12}}\wedge \g_3, 
\]
where $(v^1,v^2,v^3)$ is an orthonormal basis of $\frm_3^*$ and $(\gamma_1,\gamma_2,\gamma_3)$ is an $\SU(2)$-structure on $\frm_4$. 
Since  $d\theta \in \la \omega_1 \ra \subset \Lambda^2 \frm_4^*$, the condition $d\theta \wedge \star \varphi =0$ yields $d\theta \wedge \gamma_k=0$, 
for $k=1,2,3$, which is equivalent to $d\theta=0$.
\end{proof}

We now prove a general result ensuring that $e^5$ and the Lee form $\theta$ are $g_\f$-orthogonal.

\begin{proposition} \label{prop:integrable-rho-metric}
Let $\rho\in(\L^4\frm^*)^\H$ be an $\Ad(\H)$-invariant $4$-form such that $d\rho= \theta \wedge \rho$, 
for some $\Ad(\H)$-invariant non-zero closed $1$-form $\theta$ on $\frm$. 
Then, the bilinear map
$B_\rho : \frm^* \times \frm^* \to (\L^7 \frm^*)^{\otimes 2}$ 
satisfies $B_\rho\left(e^k, \theta\right)= 0$, for $1 \leq k \leq 5$.
\end{proposition}
\begin{proof}
Since $\theta \in  \la e^6, e^7 \ra\subset \frm^*_3$ and the $\Ad(\H)$-modules $\frm_4$ and $\frm_3$ are inequivalent, 
we have $B_{\rho}(e^k,\theta)=0$ for $1 \leq k \leq 4$. 
We now prove that $B_{\rho}(e^5,\theta)=0$. 
For this purpose, we let $\eta^1\coloneqq \theta$, $\eta^3\coloneqq e^5$, and we choose $\eta^2 \in \la e^6,e^7 \ra$ linearly independent to $\eta^1$. 
We can write
\[
\rho= \l_{1234}\omega_1^2 + \eta^{23} \wedge \b_1 + \eta^{13} \wedge \b_2 + \eta^{12}\wedge \b_3,
\] with
$\b_j= \sum_{k=0}^3 \l_j^k \o_k$, where $\l_j^k$, $\l_{1234} \in \R$. Using equations \eqref{eq:str-diff-1} and \eqref{eq:str-diff-2} we get 
\[
\begin{split}
d\rho &= (\l_{1}^1 \eta^2 + \l_2^1 \eta^1)\wedge \omega_1^2 
+ \eta^{123} \wedge (-c\l_3^2 \omega_3 + c\l_3^3 \omega_2).
\end{split}	
\]
In addition,
\[
\eta^1 \wedge \rho= \l_{1234}\eta^1 \wedge \omega_1^2 + \eta^{123} \wedge \left( \l_1^0\, \o_0 +  \l_1^1\,\o_1 +  \l_1^2\,\o_2 + \l_1^3\, \o_3 \right).
\]

Since $\theta \neq 0$, the condition $d\rho= \eta^1 \wedge \rho$ is then equivalent to the following system of equations 
\[
\l_2^1=\l_{4567}, \quad 
-c\l_3^2=\l_1^3, \quad
c\l_3^3=\l_1^2, \quad
\l_1^0=\l_1^1=0.
\]

We now compute the bilinear map $B_{\rho}$ as explained in Section \ref{SecDefForm7}. 
Consider the dual basis $(\eta_1,\eta_2,\eta_3)$  of $(\eta^1,\eta^2,\eta^3)$, and define $\widehat{\o}_1 \coloneqq e_{12}- e_{34}$, 
$\widehat{\o}_2 \coloneqq e_{13}+e_{24}$, $\widehat{\o}_3\coloneqq e_{14}+ e_{23}\in\Lambda^2\frm$.  
Moreover, let $\widehat{\o}_0 \coloneqq -\left(e_{12}+ e_{34}\right) $ if $\G=\SU(2)^2 \times \U(1)^2$ 
and $\widehat{\o}_0 \coloneqq 0$ if $\G=\SU(3)\times \U(1)^2$. 
Define $\widehat{\rho}\in\Lambda^3\frm$ via the identity 
$\iota_{\widehat{\rho}}\Omega = \rho$, where $\Omega=-e^{1234}\wedge \eta^{123}$. 
Then 
\[
\widehat{\rho}=  \l_2^1 \eta_{123} + \eta_1 \wedge \g_1 - \eta_2 \wedge \g_2 + \eta_3 \wedge \g_3,
\] 
where  $\g_1= c \l_3^3 \widehat{\o}_2 - c \l_3^2 \widehat{\o}_3$
and $\g_j= \sum_{k=0}^3{\l_j^k \widehat{\o}_k}$ for $j=2,3$. Observe that
\[
\g_1 \wedge \g_3= c\l_3^3\l_3^2 (\widehat{\o}_2^2 - \widehat{\o}_3^2) = 0.
\]
We use this last equality twice to show that $B_\rho\left(\eta^1,\eta^3\right)=0$:
\[
 \iota_{\eta^1} \widehat{\rho} \wedge \iota_{\eta^3} \widehat{\rho}\wedge  \widehat{\rho}
 = \l_2^1 \left( \eta_{23} \wedge \g_3 + \eta_{12} \wedge \g_1 \right) \wedge \widehat{\rho} = 2 \l_2^1 \eta_{123}\g_1 \wedge \g_3=0.
\]
\end{proof}

Denote by $\G'$ the semisimple part of $\G$, and consider the six-dimensional manifold $N\coloneqq{\G}'/\H \times S^1$, 
that is homogeneous under the transitive action of $\G'\times\U(1)$.   
Here $\G'/\H$ is diffeomorphic to $S^5$ when ${\G}'=\SU(3)$ and to $V^{4,2}$ when ${\G}'=\SU(2)^2$. 
We now relate the existence of an invariant strong $\G_2T$-structure on $\G/\H$ to the existence of an invariant $\SU(3)$-structure on $N$ 
satisfying equations \eqref{G2TNS} and \eqref{G2TNSdT}.

\begin{lemma} \label{lem:reduction-su(3)}
If there exists an invariant strong $\G_2T$-structure on $\G/\H$, 
then there exists an invariant $\SU(3)$-structure on $N$ satisfying equations \eqref{G2TNS} and \eqref{G2TNSdT} with $w_2^+=0$.
\end{lemma}
\begin{proof}
Assume there exists an $\Ad(\H)$-invariant strong $\G_2T$-structure $\f$ on $\frm$, and let $g_\f$ be the induced metric.

We first consider the case when $\varphi$ is coclosed, namely $\theta=0$. 
Let $\eta_1 \in \la e_6,e_7\ra$ be a unit-length vector perpendicular to $e_5$. 
Its metric dual $\eta_1^\sharp = g_\f(\eta_1,\cdot)$ lies in $\la e^6, e^7 \ra$, and we can consider the six-dimensional $\Ad(\H)$-invariant 
subspace $\frn \coloneqq \ker(\eta_1^\sharp)= \frm' \oplus \la \eta_2 \ra$ of $\frm$, 
for some unit-length vector $\eta_2 \in \la e_6,e_7\ra$. 
Then, $\varphi= \eta_1^\sharp \wedge \omega + \psip$, for some $\Ad(\H)$-invariant invariant $\SU(3)$-structure $(\omega, \Psi)$ on $\frn$, 
and the decomposition $\frm = \frn \oplus \la \eta_1 \ra$ is $g_\f$-orthogonal. 
If we let $T^2 \coloneqq \R \la \eta_1, \eta_2 \ra/\mathbb{Z} \la \eta_1, \eta_2 \ra$, then  
the strong $\G_2T$-structure on $\frm$ induces an invariant strong $\G_2T$-structure on the homogeneous space
$
{\G}'/\H \times T^2, 
$
whose metric makes $N= {\G}'/\H \times \R \la \eta_2 \ra/ \Z \la \eta_2 \ra$ and $S^1= \R \la \eta_1 \ra/\Z\la \eta_1 \ra$ perpendicular. 
Proposition \ref{prop:product-closedT} implies then that the pair $(\omega, \Psi)$ satisfies equations \eqref{G2TNS} and \eqref{G2TNSdT} 
with $w_2^+=0$ and $\lambda=0$.

We now assume that $\theta \neq 0$. Lemma \ref{lem:lee-closed} guarantees that $\theta \in \la e^6,e^7 \ra$.  
We take a unit-length $1$-form $\eta^1$ with $\theta= \lambda \eta^1$, $\lambda>0$, and  
a unit-length $1$-form $\eta^2= \lambda_6 e^6 + \lambda_7 e^7$ perpendicular to $\eta^1$.
Then, Proposition \ref{prop:integrable-rho-metric} ensures that $\frn^*= \la e^1, \dots , e^5, \eta^2 \ra \subset \frm^*$  
is perpendicular to $\eta^1$, and we can write
$
\varphi= \eta^1 \wedge \omega + \psip,
$
where the pair $(\omega, \Psi)$ determines an $\Ad(\H)$-invariant $\SU(3)$ structure on the $\Ad(\H)$-invariant subspace of $\frm$  
\[
(\frn^*)^* = \frm' \oplus \la \lambda_6 e_6 + \lambda_7 e_7 \ra. 
\]
This pair induces an invariant $\SU(3)$-structure on ${\G}'/\H \times S^1$, with $S^1= \R \la \eta_2 \ra/ \Z \la \eta_2 \ra$. 
Taking into account that $d\eta_1=0$, $d\omega \in \Lambda^3 \frn^*$, $d\psi_{\pm} \in \Lambda^4 \frn^*$, 
we can argue as in the proof of Proposition \ref{prop:product-closedT} and conclude that 
the SU(3)-structure $(\omega, \Psi)$ satisfies equations \eqref{G2TNS} and \eqref{G2TNSdT} with $w_2^+=0$. 
\end{proof}

Since $b_3(S^5\times S^1)=0$, Corollary \ref{Corb3N} allows us to conclude the discussion for the case $6$ of Table \ref{TabRed}.  
\begin{corollary}
There are no invariant strong $\G_2T$-structures on $\G/\H$ when $\G=\SU(3)\times \U(1)^2$ and $\H=\SU(2)\subset\SU(3)$.
\end{corollary}

Since $b_3(V^{4,2})=1$, we cannot rule out the case $7$ of Table \ref{TabRed} using the same argument as before. 
However, by Lemma \ref{lem:reduction-su(3)}, this case can be discarded by proving that there are no invariant $\SU(3)$-structures 
satisfying the equations \eqref{G2TNS} and \eqref{G2TNSdT} with $w_2^+=0$ on the homogeneous space $N = \G'\times\U(1)/ \H$, 
where $\G'=\SU(2)^2$ and $\H\cong\U(1)$ is diagonally embedded into $\G'$ 
as described in \eqref{U1diag}. 
Again, this boils down to showing that there are no $\Ad(\H)$-invariant $\SU(3)$-structures on 
$\frn = \frm'\oplus\R$ satisfying the required conditions. 
We consider the basis $(e_1,\dots,e_5)$ of  $\frm' \subset \frsu(2)\oplus \frsu(2)$ previously defined, and a generator $e_6 \in \R$, 
and we keep on using the notations introduced before. 

Recall that the conditions \eqref{G2TNS} and $w_2^+=0$ imply that the $\Ad(\H)$-invariant $\SU(3)$-structure $(\omega,\Psi)$ on $\frn$ satisfies 
\begin{equation}\label{RecapEqLast1}
d\omega	= -\frac{3}{2}w_1^-\, \psi_+ + \frac{3}{2}\,w_1^+ \psi_- + w_3,\quad
d\psi_+	= w_1^+ \omega^2,\quad
d\psi_-	=  w_1^-\, \omega^2, 
\end{equation}
and that \eqref{G2TNSdT} is equivalent to 
\begin{equation}\label{RecapEqLast2}
\delta_g w_3 =  \left( \left( w_1^+\right)^2 + \left( w_1^- \right)^2\right) \omega. 
\end{equation}

The generic $\Ad(\H)$-invariant 2-form $\frn$ can be written as follows
\begin{equation}\label{eq:omega}
\omega= a_0 \o_0 + a_1 \o_1 + a_2 \o_2 + a_3 \o_3 + a_4 e^{56},
\end{equation}
for some real constants $a_0,\ldots,a_4$. 
Moreover, in \eqref{eq:omega} we can additionally assume that $a_2=0$. 
Indeed, the one-dimensional torus $\U\coloneqq \exp(\R e_5)$ centralizes the isotropy $\H$, and for every $u=\exp(te_5)\in\U$ we can consider the
$\G'\times\U(1)$-equivariant diffeomorphism $\tau_u$ of $N$ given by $\tau_u(x\H) = xu\H$, where $x\in \G'\times\U(1)$. 
Then, under the usual identification between invariant forms on $\G'\times\U(1)/\H$ and $\Ad(\H)$-invariant forms on $\frn$, 
we have
\[
\tau_u^*\omega = a_0 \o_0 + a_1 \o_1 + a_2' \o_2 + a_3' \o_3 + a_4 e^{56},
\] 
where $a_2 ' = \cos(2t)\,a_2-\sin(2t)\,a_3$ and $a_3' = \sin(2t)\,a_2+\cos(2t)\,a_3$, and we obtain $a_2'=0$ for a suitable choice of $t$. 

 We have the following useful result.  
\begin{lemma}\label{prop:w2-vanish}
Let $\omega$ be an $\Ad(\H)$-invariant $2$-form on $\frn$ given by \eqref{eq:omega} with $a_2=0$. 
Then, $\omega$ cannot be symplectic. Moreover, if $\omega$ is non-degenerate, then it satisfies the condition $d\omega\W\omega=0$ 
if and only if $a_1=0$. 
\end{lemma}
\begin{proof}
We have 
\[
\omega^3 = 6\,a_4\left(a_0^2-a_1^2-a_3^2\right) e^{123456}. 
\]
Moreover, using \eqref{eq:str-diff-2}, we obtain
\[
d\omega = 4a_3\, \omega_2\W e^5 - a_4\, \omega_1\W e^6,
\]
whence we get 
\[
d\omega\W\omega= 2 a_1 a_4  e^{12346}. 
\]
Thus, if $\omega$ is non-degenerate, then $a_4\neq0$. In particular, $\omega$ cannot be closed  
and the condition $d\omega\W\omega=0$ is equivalent to $a_1=0$.  
\end{proof}

By the previous lemma and (\ref{prodclosed3}) of Proposition \ref{prop:product-closedT}, 
if $(\omega,\Psi)$ is an $\Ad(\H)$-invariant $\SU(3)$-structure on $\frn$ solving \eqref{RecapEqLast1}, 
then at least one of the intrinsic torsion forms $w_1^+$ and 
$w_1^-$ must be non-zero, as otherwise the $\SU(3)$-structure should be torsion-free. 
We can then argue as in the proof of Proposition \ref{prop:product-str} and conclude that if $\frn$ admits an $\Ad(\H)$-invariant $\SU(3)$-structure 
satisfying \eqref{RecapEqLast1} and \eqref{RecapEqLast2}, then it must admit an $\Ad(\H)$-invariant double half-flat $\SU(3)$-structure with 
$w_1^-\neq0$ solving
\[
\Delta_g \omega =  4(w_1^-)^2 \omega. 
\]
Recall that this last condition is equivalent to \eqref{RecapEqLast2} by Corollary \ref{corDHF}. 
Thus, to discard case $7$ of Table \ref{TabRed}, it is sufficient showing that $\frn$ does not admit any such structure. 

For our aim, it is convenient giving a parametrization of all $\Ad(\H)$-invariant double half-flat $\SU(3)$-structures on $\frn$ first. 
Recall from \cite{Hitchin} that a pair $(\omega,\psip)\in(\Lambda^2\frn^*)^\H\times(\Lambda^3\frn^*)^\H$ 
defines an $\SU(3)$-structure on $\frn$ if and only if the following hold 
\begin{enumerate}[(i)]
\item\label{ifirst} $\omega$ is non-degenerate, namely $\omega^3\neq0$;
\item $\psip$ is a {\em negative stable} 3-form, namely a stable 3-form whose associated quartic polynomial satisfies $\lambda(\psip)<0$. 
Here, $\lambda(\psip) \coloneqq \frac16\mathrm{tr}(K_{\psip}^2)$, 
where $K_{\psip}\in\End(\frn)$ is defined as follows. 
Let $A:\Lambda^5\frn^*\rightarrow \frn\otimes \Lambda^6\frn^*$ be the isomorphism induced by the wedge product, then 
$K_{\psip}(v)\otimes \omega^3 = A(\iota_v\psip\W\psip)$, for all $v\in\frn$. 
In particular, $K_{\psip}^2 = \lambda(\psip)\mathrm{Id}_\frn$, so that $(\omega,\psip)$ determines an almost complex structure 
\[
J:\frn\rightarrow\frn,\quad J = \frac{1}{\sqrt{-\lambda(\psip)}} \, K_{\psip};
\]
\item $\psip$ is primitive with respect to $\omega$, i.e., $\psip\W\omega=0$. 
This is equivalent to $\omega$ being of type $(1,1)$ with respect to $J$, namely $\omega(J\cdot,J\cdot)=\omega$;
\item the symmetric bilinear form $g \coloneqq \omega(\cdot,J\cdot)$ is positive definite;
\item\label{ilast} the normalization condition $\psip\W\psim = \tfrac23\omega^3$ holds, where $\psim\coloneqq J\psip$. 
\end{enumerate}
Using this characterization and the aid of the software Maple 2021 for the computations, we obtain the following.

\begin{proposition}\label{prop:double-half-flat-charact}
There is a $4$-parameter family of $\Ad(\H)$-invariant double half-flat $\SU(3)$-structures on $\frn$ that is defined by the pair 
\[
\omega= a_3\, \o_3 + a_4\, e^{56},\quad
\psip	= (b_0 \o_0  + b_2 \o_2) \wedge e^5 + (c_0 \o_0 + c_1 \o_1 )\wedge e^6,
\] 
where the real parameters $b_0,b_2,c_0,c_1$ satisfy the constraints
\[
b_0^2c_1^2 +b_2^2\left(c_0^2-c_1^2\right)<0,\quad c_1\left(c_0+c_1\right)>0, 
\]
and the components of $\o$ are given by 
\begin{equation}\label{a3a4}
a_3^3 = \varepsilon\, \frac{ \,c_1 \left(b_0^2-b_2^2\right)}{8\, b_2\left(c_0^2-c_1^2\right)}\,\sqrt{-\lambda(\psip)}, \qquad
a_4^3 = -2\,\varepsilon\, \frac{b_2^2\left(c_0^2-c_1^2\right)^2}{c_1^2\left(b_0^2-b_2^2 \right)^2}\,\sqrt{-\lambda(\psip)}, 
\end{equation}
with $\varepsilon\in\{\pm1\}$. 
\end{proposition}

\begin{proof}
Let us consider a generic $\Ad(\H)$-invariant $2$-form on $\frn$
\[
\omega= a_0 \o_0 + a_1 \o_1 + a_3 \o_3 + a_4 e^{56},
\]
and a generic $\Ad(\H)$-invariant $3$-form 
\[
\psip=  \left( b_0 \o_0 + b_1 \o_1 + b_2 \o_2 + b_3 \o_3 \right)\wedge e^5 +  \left( c_0 \o_0 + c_1 \o_1 + c_2 \o_2 + c_3 \o_3 \right)\wedge e^6, 
\]
where $a_k,b_k,c_k$ are real parameters. 
We discuss for which values of these parameters the pair $(\o,\psip)$ defines a double half-flat $\SU(3)$-structure on $\frn$. 
This happens if and only if $(\o,\psip)$ satisfies the conditions (\ref{ifirst})-(\ref{ilast}) recalled before the statement of this proposition and 
\[
d\o\W\o=0,\quad d\psip=0,\quad d\psim = w_1^-\,\o^2.
\]

By Lemma \ref{prop:w2-vanish}, we know that $\omega$ is non-degenerate and satisfies the condition $d\omega\W\omega=0$ if 
and only if 
\[
\frac{\omega^3}{6} = a_4\left(a_0^2-a_3^2\right) e^{123456} \neq0, \qquad a_1=0.  
\]
We denote by $\varepsilon\in\{\pm1\}$ the sign of $a_4\left(a_0^2-a_3^2\right)$, 
so that $\omega$ induces the same orientation as $\varepsilon\,e^{123456}$.

Using \eqref{eq:str-diff-2}, we obtain 
\[
d\psip = -b_1\,\o_1^2 -4 c_2\, \omega_3 \wedge e^{56} + 4c_3\, \omega_2 \wedge e^{56},
\]
whence it follows that $\psip$ is closed if and only if $b_1=0=c_2=c_3$. 

Now, the endomorphism $K_{\psip}\in\End(\frn)$ is represented by the following matrix with respect to the basis $(e_1,\ldots,e_6)$
\[
2\varepsilon
\left(
\begin{array}{cccccc}
 b_{0} c_{1} & 0 &  b_{3} \left(c_{1}-c_{0}\right) &  b_{2} (c_{0}-  c_{1}) & 0 & 0 
\\
 0 &  b_{0} c_{1} & b_{2} \left(c_{1}-c_{0}\right) &  b_{3} (c_{1}- c_{0}) & 0 & 0 
\\
 -b_{3}  \left(c_{0}+c_{1}\right) & - b_{2} (c_{0}+ c_{1}) & - b_{0} c_{1} & 0 & 0 & 0 
\\
 b_{2}  \left(c_{0}+c_{1}\right) & - b_{3} (c_{0} + c_{1}) & 0 & - b_{0} c_{1} & 0 & 0 
\\
 0 & 0 & 0 & 0 & - b_{0} c_{0} &   c_{1}^2-c_{0}^2 
\\
 0 & 0 & 0 & 0 &  b_{0}^{2}- b_{2}^{2}- b_{3}^{2} &  b_{0} c_{0} 
\end{array}
\right).
\]
Using this expression, we compute the quartic polynomial $\lambda(\psip)=\tfrac16\mathrm{tr}(K_{\psip}^2)$, 
and we see that $\psip$ is negative stable if and only if 
\[
\lambda(\psip) = 4\left(b_0^2c_1^2 + \left(b_2^2+b_3^2\right)\left(c_0^2-c_1^2\right)\right) <0. 
\]
This implies in particular that $c_0^2-c_1^2 <0$ and thus $c_1\neq0$. 
We then obtain the almost complex structure $J = \frac{1}{\sqrt{-\lambda(\psip)}} \, K_{\psip}$, and we compute 
\[
\psim 	\coloneqq J\psip 
= \left(p_0\, e^{12} + p_1\, e^{34} + p_2\, \o_2 + p_3\, \o_3 \right)\wedge e^5 +  \left(q_0\,e^{12} + q_1\,e^{34} +q_2\,\o_2+ q_3\, \o_3 \right)\wedge e^6,
\] 
where $p_j$ and $q_j$ depend on $b_k$ and $c_k$. 

Let us now consider the condition $d\psim= w_1^-\omega^2$. 
We have 
\[
\omega^2 = \left(a_3^2-a_0^2\right)\omega_1^2 + 2a_0a_4\,\o_1\W e^{56} + 2a_3a_4\,\o_3\W e^{56},
\]
while 
\[
d\psim = \frac12 (p_1-p_0)\, \o_1^2 - 4 q_2\, \o_3 \wedge e^{56} + 4 q_3\, \o_2 \wedge e^{56}, 
\]
where 
\[
\frac12 (p_1-p_0) =\frac{2\,\varepsilon \,c_1\left(b_0^2-b_2^2-b_3^2\right)}{\sqrt{-\lambda(\psip)}},\quad 
q_2 = \frac{2\,\varepsilon\,b_2\left(c_0^2-c_1^2\right)}{\sqrt{-\lambda(\psip)}},\quad
q_3 = \frac{2\,\varepsilon\,b_3\left(c_0^2-c_1^2\right)}{\sqrt{-\lambda(\psip)}}.
\]
Since $\omega$ has no component along $\o_2 \wedge e^{56}$, we must have $q_3=0$, which implies $b_3=0$, as $c_0^2-c_1^2 <0$. 
On the other hand, $d\psim$ has no component along $\o_1\W e^{56}$ and $\omega$ is non-degenerate, so we must have $a_0=0$. 
Consequently, $\omega^3 = -6\,a_3^2\,a_4\,e^{123456}$, so that $a_3a_4\neq0$ and $\varepsilon$  is the sign of $-a_4$. 
Moreover, $\lambda(\psip) = 4\left( b_0^2c_1^2 + b_2^2\left(c_0^2-c_1^2\right)\right) =  4\left( b_2^2c_0^2 + c_1^2 \left(b_0^2 - b_2^2\right)\right)<0$,  
which implies $b_0^2 - b_2^2<0$, $c_0^2-c_1^2<0$ and, in particular, $b_2c_1\neq0$. 
The condition $d\psim= w_1^-\omega^2$ gives then
\[
a_3^2\,w_1^- = \frac{2\, \varepsilon\,c_1 \left(b_0^2-b_2^2\right)}{\sqrt{-\lambda(\psip)}},\qquad  
a_3a_4\,w_1^- = -\frac{4\,\varepsilon\,b_2\left(c_0^2-c_1^2\right)}{\sqrt{-\lambda(\psip)}}, 
\]
whence we get 
\begin{equation}\label{Lasta3}
a_3 = -\frac{ a_4\,c_1 \left(b_0^2-b_2^2\right)}{2\, b_2\left(c_0^2-c_1^2\right)},\qquad
w_1^- = \frac{ 8\, \varepsilon\, b_2^2\left(c_0^2-c_1^2\right)^2}{ a_4^2\,c_1 \left(b_0^2-b_2^2\right)\sqrt{-\lambda(\psip)}} . 
\end{equation}

After these computations, the expressions of $\omega$ and $\psip$ reduce to 
\[
\o = a_3\, \o_3 +a_4\, e^{56}, \quad 
\psip = (b_0 \o_0  + b_2 \o_2) \wedge e^5 + (c_0 \o_0 + c_1 \o_1 )\wedge e^6,
\]
where $a_3$ is given by \eqref{Lasta3}. 
These forms satisfy the compatibility condition $\omega\W\psip=0$, and we can then consider the 
symmetric $2$-covariant tensor $g\coloneqq \omega(\cdot,J\cdot)$. 
Its scaling  $\frac{\sqrt{-\lambda(\psip)}}{2}\,g$ is represented by the following matrix with respect to the basis $(e_1,\ldots,e_6)$
\[
\varepsilon
\begin{pmatrix}
\left(c_{0}+c_{1}\right) b_{2} a_{3} & 0 & 0 & -b_{0} c_{1} a_{3} & 0 & 0 
\\
 0 & \left(c_{0}+c_{1}\right) b_{2} a_{3} & b_{0} c_{1} a_{3} & 0 & 0 & 0 
\\
 0 & b_{0} c_{1} a_{3} &  \left(c_{1}-c_{0}\right) b_{2}a_{3} & 0 & 0 & 0 
\\
 -b_{0} c_{1} a_{3} & 0 & 0 & \left(c_{1}-c_{0}\right)  b_{2}a_{3} & 0 & 0 
\\
 0 & 0 & 0 & 0 & \left(b_{0}^{2}-b_{2}^{2}\right) a_{4} & b_{0} c_{0} a_{4} 
\\
 0 & 0 & 0 & 0 & b_{0} c_{0} a_{4} & a_{4} \left(c_{0}^{2}-c_{1}^{2}\right) 
\end{pmatrix}.
\]
The additional constraint $c_1\left(c_0+c_1\right) >0$ ensures then that $g$ is positive definite. 

We are left with the normalization condition, which gives 
\[
2\,\varepsilon\,\sqrt{-\lambda(\psip)}\,e^{123456} = \psip \W\psim = 
\frac23\,\omega^3 = - \frac{a_4^3\,c_1^2\left(b_0^2-b_2^2 \right)^2}{b_2^2\left(c_0^2-c_1^2\right)^2}e^{123456}.  
\]
From this and \eqref{Lasta3}, we finally obtain the expressions of $a_3^3$ and $a_4^3$ given in \eqref{a3a4}. 
\end{proof}

We are now ready to conclude the discussion. 

\begin{proposition}
There are no $\Ad(\H)$-invariant double half-flat $\SU(3)$-structures on $\frn$ satisfying the condition $\Delta_g \o= 4 (w_1^-)^2 \o$. 
Therefore, there are no invariant strong $\G_2T$-structures on the homogeneous space $\G/\H$, where
$\G =\SU(2)^2\times \U(1)^2$ and $\H\cong\U(1)$ is diagonally embedded into $\SU(2)^2$ as in \eqref{U1diag}.   
\end{proposition}
\begin{proof}
The equation $\Delta_g \o= 4 (\o_1^-)^2 \o$ is equivalent to $d\star_6 d \o= - 2 (w_1^-)^2 \o^2$. 
We first compute the expression of $d\star_6 d \omega$ for the generic double half-flat $\SU(3)$-structure $(\omega,\Psi)$ 
described in Proposition \ref{prop:double-half-flat-charact}. 

Let us consider $\o = a_3\o_3 +a_4 e^{56}$, where $a_3$ and $a_4$ are given in \eqref{a3a4}. 
We have that $d \omega = 4 a_3\, \omega_2 \wedge e^5 - a_4\, \omega_1 \wedge e^6$. 
The Hodge operator $\star_6$ is determined by the metric described in the proof of Proposition  \ref{prop:double-half-flat-charact} 
and the volume form $2\varepsilon\sqrt{-\lambda(\psip)}\,e^{123456}$. We compute 
\[
\star_6 d \omega = \left({r_0\,e^{12} +r_1\,e^{34}} + r_2 \o_2 \right) \wedge e^5 + \left({s_0\,e^{12} + s_1\,e^{34}} + s_2 \o_2\right) \wedge e^6, 
\]
where the components $s_i$ and $r_i$ depend on $b_k$ and $c_k$. 
We then have
\[
d \star_6 d \omega = { \left(r_0- r_1\right) } e^{1234} -4 s_2\, \o_3 \wedge e^{56}, 
\]
where the coefficients are 
\[
\begin{split}
r_0-r_1 	&= 	-\frac{32\,b_2^2\left(c_0^2-c_1^2\right)}{a_4^2\,c_1^2\left(b_0^2-b_2^2\right)\lambda(\psip)}
			\left(5\,b_0^2\,c_0^2\,c_1^2 - b_0^2\,c_1^4 - b_2^2\,c_0^4 + b_2^2\,c_1^4\right),\\
s_2 		&= 	-\frac{32\,b_2 \left(c_0^2-c_1^2\right)^2}{a_4^2\,c_1 \left(b_0^2-b_2^2\right)^2 \lambda(\psip)}
			\left( b_0^4\,c_1^2 - 2\,b_0^2\,b_2^2\,c_0^2 + b_2^4\,c_0^2 - b_2^4\,c_1^2 \right),
\end{split}
\]
and the expressions of $a_4$ and $\lambda(\psip)$ are provided in Proposition \ref{prop:double-half-flat-charact}.

In addition, using the expression \eqref{Lasta3} of $a_3$ and $w_1^-$, we get
\[
2 (w_1^- )^2\,\omega^2 = 	\frac{64\, b_2^2 \left(c_0^2-c_1^2\right)^2}{a_4^2\, \lambda(\psip)}\,  e^{1234} 
					+ \frac{128\, b_2^3\, \left(c_0^2-c_1^2\right)^3}{a_4^2\,c_1 \left(b_0^2-b_2^2\right)  \lambda(\psip)}\, \o_3 \wedge e^{56}.
\]
Therefore, the identity $d\star_6 d \o= - 2 (w_1^-)^2\, \o^2$ holds if and only if 
\[ 
(r_0-r_1) + \frac{64\, b_2^2 \left(c_0^2-c_1^2\right)^2}{a_4^2\, \lambda(\psip)} = 0
= s_2 - \frac{32\, b_2^3\, \left(c_0^2-c_1^2\right)^3}{a_4^2\,c_1 \left(b_0^2-b_2^2\right)  \lambda(\psip)}. 
\]  
Taking into account the expression \eqref{a3a4} of $a_4^3$, we get
\[ 
0 = \frac{1}{a_4} \left(s_2 - \frac{32\, b_2^3\, \left(c_0^2-c_1^2\right)^3}{a_4^2\,c_1 \left(b_0^2-b_2^2\right)  \lambda(\psip)}\right) =
\varepsilon\, \frac{16\,c_1\,b_0^2\left(\lambda(\psip)-8\,b_2^2 c_0^2\right)}{b_2\left(-\lambda(\psip)\right)^{3/2}}. 
\]
Since $\lambda(\psip)<0$ and $c_1\neq0$, we must have $b_0=0$. This implies $\lambda(\psip)= 4 b_2^2\left(c_0^2-c_1^2\right)$ and
\[ 
0 = \frac{1}{a_4} \left((r_0-r_1) + \frac{64\, b_2^2 \left(c_0^2-c_1^2\right)^2}{a_4^2\, \lambda(\psip)}  \right) = 
\frac{4\,\varepsilon\,b_2^2}{\sqrt{- \lambda(\psip)}}. 
\]
Therefore, there is no solution to the equation $d\star d \o= - 2 (w_1^-)^2 \o^2$ .

Corollary \ref{corDHF} and Lemma \ref{lem:reduction-su(3)} imply that $\G/\H$ does not admit any invariant strong $\G_2T$-structure. 
\end{proof}

We are left with the spaces $8$-$9$-$10$ of Table \ref{TabRed}. 
We now discard the existence of invariant strong $\G_2T$-structures on them by taking into account Corollary \ref{Corb3N}.

\begin{proposition}
There are no invariant strong $\G_2T$-structures on the homogeneous spaces 
$F(1,2)\times S^1 = \SU(3)\times \U(1)/\U(1)^2$, 
$\mathbb{C}P^3\times S^1 = \Sp(2)\times\U(1)/\Sp(1)\times\U(1)$ and 
$S^6 \times S^1 = \G_2\times\U(1)/\SU(3)$.
\end{proposition}
\begin{proof}
Let $\G/\H$ be one of these spaces. Then, $\G=\G'\times \U(1)$, $\H \subset \G'$,  and 
$\H$ does not fix any non-zero vector on the tangent space to $\G'/\H$. 
Therefore, $\G/\H\cong\G'/\H \times \U(1)$, and invariant metrics on $\G/\H$ make $\G'/\H$ perpendicular to $\U(1)$. 
Consequently, invariant $\G_2$-structures are of the form $\varphi= \omega \wedge \eta + \psi_+$, 
for some invariant $\SU(3)$-structure $(\omega,\Psi)$ on $\G'/\H$ and some invariant $1$-form $\eta$ on $\U(1)$. 
In addition, the Lee form of any invariant $\G_2$-structure must be proportional to $\eta$. 
According to Corollary \ref{Corb3N}, if there were an invariant strong $\G_2T$-structure on $\G/\H$, then $b_3(\G'/\H)>0$. 
However, the spaces $\G'/\H$ correspond to $F(1,2)$, $\mathbb{C}P^3$, and $S^6$, and these have vanishing third Betti number. 
\end{proof}

%
\section{Examples of $\G_2T$-structures evolving under the Laplacian coflow}\label{CoflowSect}
We now focus on the examples of left-invariant strong $\G_2T$-structures on $\G=\SU(2)\times \U(1)^4$ and $\G=\SU(2)\times\SU(2)\times \U(1)$  
considered in the proofs of propositions \ref{S3T4G2T} and \ref{S3S3S1G2T}, respectively, and we investigate their behavior under the Laplacian coflow 
\[
\frac{\partial}{\partial t} \psi(t) = \Delta_{\psi(t)} \psi(t) = (dd^{*_t}+d^{*_t}d)\psi_t = d\star_t d\star_t \psi(t) - \star_td\star_td\psi(t). 
\]
Here, $\psi(t)$ is a one-parameter family of definite $4$-forms and $\star_t$ is the Hodge operator determined by the induced metric and 
an orientation that is fixed throughout the flow. 
A natural choice for the latter is the orientation determined by a given $\G_2$-structure $\f_0$ such that $\psi(0)=\star_0\f_0$. 

Since we are interested in left-invariant solutions to the flow, it is sufficient working on the Lie algebra $\frg$ of the Lie group $\G$. 
In this setting, the Laplacian coflow is equivalent to a system of ODEs for certain unknown functions, and  
the choice of the initial datum $\psi(0) = \psi_0$ gives an initial value problem for these functions.  
We will determine the explicit solution to this problem in both examples, 
showing that the solution to the Laplacian coflow starting at the strong $\G_2T$-structure $\psi_0$ under exam exists (at least) for all positive times 
and it defines a strong $\G_2T$-structure.

\subsection{The example on $\mathbf{{SU}(2)\times {U}(1)^4}$}
As in the proof of Proposition \ref{S3T4G2T}, 
we choose a basis $(e_1,\ldots,e_7)$ of $\frg= \R^3\oplus\frsu(2)\oplus\R$ whose dual basis $(e^1,\ldots,e^7)$ satisfies the structure equations 
\[
de^1=de^2=de^3=de^7=0,\quad de^4=e^{56},\quad de^5=-e^{46},\quad de^6=e^{45}. 
\] 
Let us consider the strong $\G_2T$-structure 
\[
\varphi_0 = e^{123}+ e^1\wedge (e^{45}+e^{67}) + e^2\wedge (e^{46}-e^{57}) - e^3\wedge (e^{47}+e^{56}), 
\]
inducing the metric $g_0 = \sum_{i=1}^7 e^i\odot e^i$ and the orientation $dV_0 = e^{1234567}$. The dual $4$-form is
\[
\psi_0 = \star_0\f_0 = e^{4567}+e^{23}\W(e^{45}+e^{67}) - e^{13}\W(e^{46}-e^{57}) -e^{12}\W(e^{47}+e^{56}), 
\]
and it satisfies the equation 
\[
d\psi_0 = \theta_0\W\psi_0,
\]
where the Lee form is $\theta_0=e^7$. Recall that the torsion $3$-form of $\f_0$ is $T_0=e^{456}$ 
and that this strong $\G_2T$-structure solves the twisted $\G_2$ equation \eqref{eqn:strom7}. 

We now look for the solution to the Laplacian coflow starting at $\psi_0$, namely  
\[
\begin{cases}
\frac{\partial}{\partial t} \psi(t) = \Delta_{\psi(t)} \psi(t), \\
\psi(0)=\psi_0.
\end{cases}
\]
We consider the following Ansatz
\[
\psi(t) =  a_1(t) e^{4567}+a_2(t)e^{2345} + a_3(t) e^{2367} -a_4(t) e^{1346} + a_5(t)e^{1357} -a_6(t)e^{1247} -a_7(t) e^{1256}, 
\]
where for $k=1,\ldots,7$, $a_k(t)$ is an unknown function such that $a_k(0)=1$. 
Henceforth, we will omit the dependence on $t$ of these functions for brevity. 

Choosing the orientation $\Omega \coloneqq dV_0$, we obtain  
\[
\begin{split}
&B_{\psi(t)}(e^1,e^1) = a_1a_2a_3\, \Omega^{\otimes2}, \quad 
B_{\psi(t)}(e^2,e^2) = a_1a_4a_5\, \Omega^{\otimes2},\quad 
B_{\psi(t)}(e^3,e^3) = a_1a_6 a_7 \, \Omega^{\otimes2},\\
&B_{\psi(t)}(e^4,e^4) = a_3a_5 a_7 \, \Omega^{\otimes2},\quad 
B_{\psi(t)}(e^5,e^5) = a_3a_4 a_6 \, \Omega^{\otimes2},\quad 
B_{\psi(t)}(e^6,e^6) = a_2a_5 a_6 \, \Omega^{\otimes2},\\
&B_{\psi(t)}(e^7,e^7) = a_2a_4 a_7 \, \Omega^{\otimes2}, \quad B_{\psi(t)}(e^i,e^j)=0, \mbox{ if } i\neq j. 
\end{split}
\]
Thus, $\psi(t)$ is a definite $4$-form as long as the functions $a_k$ are positive. 
In such a case, one can easily check that $\psi(t)$ defines a $\G_2T$-structure with Lee form $\theta(t) = \lambda(t) e^7$, 
for some positive function $\lambda(t)$ such that $\lambda(0)=1$,  if and only if 
\[
a_3 = \lambda\,a_2,\quad a_5 = \lambda\,a_4,\quad a_6 = \lambda\, a_7.
\]  
From the expression of $B_{\psi(t)}$ it follows that the metric induced by $\psi(t)$ is 
\[
\begin{split}
g_{\psi(t)}	&=	\frac{a_4a_7}{a_2}\sqrt{\frac{\lambda}{a_1}}e^1\odot e^1 + \frac{a_2a_7}{a_4}\sqrt{\frac{\lambda}{a_1}}e^2\odot e^2 
			+\frac{a_2a_4}{a_7}\sqrt{\frac{\lambda}{a_1}}e^3\odot e^3 \\
		&\quad	+  \sum_{k=4}^6 \sqrt{\frac{a_1}{\lambda}}e^k\odot e^k 
			+ \lambda\sqrt{\lambda a_1} e^7\odot e^7. 
\end{split} 
\]
The Hodge operator $\star_t$ is determined by $g_{\psi(t)}$ and the volume form $\sqrt{\det(g_{\psi(t)})}\Omega$. 
Now, a straightforward computation shows that  
\[
\Delta_{\psi(t)}\psi(t) 	= 	a_2\sqrt{\frac{\lambda}{a_1}}\left(e^{2345}+\lambda\, e^{2367}\right) 
					- a_4\sqrt{\frac{\lambda}{a_1}}\left(e^{1346}-\lambda\, e^{1357}\right) 
					- a_7\sqrt{\frac{\lambda}{a_1}}\left(e^{1256}+\lambda\, e^{1247}\right). 
\]
Comparing this expression with that of $\psi(t)$, we see that the Laplacian coflow is equivalent to the following initial value problem
\[
\left\{
\begin{split}
&a_1'=0,\\ 
&a_k' = a_k\,\sqrt{\frac{\lambda}{a_1}},\quad k=2,4,7,\\
&(\lambda a_k)' = \lambda\, a_k\,\sqrt{\frac{\lambda}{a_1}},\quad k=2,4,7,\\
&a_1(0)=1=a_2(0)=a_3(0)=a_4(0)=\lambda(0), 
\end{split}\right.
\]
where the prime symbol denotes the derivative with respect to $t$. 
The solution to this system is easily seen to be
\[
a_1 \equiv 1,\quad \lambda \equiv 1,\quad  a_2 = a_4 = a_7 = \exp(t). 
\]
Therefore, the solution to the Laplacian coflow starting at $\psi_0$ is given by the definite $4$-form 
\[
\psi(t) = e^{4567}+\exp(t)\left(e^{2345}+e^{2367} -e^{1346}+e^{1357} - e^{1247}-e^{1256}\right).
\]
This exists for all real times and defines a strong $\G_2T$-structure with constant Lee form $\theta(t) \equiv \theta_0= e^7$ and constant
torsion $3$-form $T(t)\equiv T_0 = e^{456}$. Moreover, it solves the twisted $\G_2$ equation \eqref{eqn:strom7}, since the definite $3$-form
\[
\f(t) = \star_t\psi(t) = \exp\left(\frac32 t\right)e^{123}+ \exp\left(\frac{t}{2}\right)\left(e^{145}+e^{167}+ e^{246}-e^{257} - e^{347}-e^{356}\right), 
\]
satisfies the equation $d\f(t)\W\f(t)=0$. 

\smallskip

Notice that, for every $t\in\R$, the global frame 
\[
e_k(t) \coloneqq \exp\left(\frac{t}{2}\right) e_k,~k=1,2,3,\quad e_k(t) \coloneqq e_k,~k=4,5,6,7,
\]
is adapted to the $\G_2$-structure $\f(t)$, and   
the manifold $\SU(2)\times\U(1)^4$ splits as the Riemannian product of the $3$-torus $T^3=\U(1)^3$ endowed with the metric 
$g_{\f(t)}|_{T^3} = \exp(t)  \sum_{k=1}^3e^k\odot e^k$ and the homogeneous Hopf surface $X^4= \SU(2)\times\U(1)$ endowed with the metric 
$g_{\f(t)}|_{X^4} = \sum_{k=4}^7e^k\odot e^k$.  
In particular, the $3$-torus is calibrated by $\f(t)$, i.e., it is an associative submanifold of $\SU(2)\times\U(1)^4$, 
and its volume increases exponentially along the flow. 
On the other hand, the Hopf surface $X^4$ is calibrated by $\psi(t)$, i.e., it is a coassociative submanifold of $\SU(2)\times\U(1)^4$, 
and its volume is fixed along the flow.

\subsection{The example on $\mathbf{{SU}(2)\times SU(2)\times  {U}(1)}$}
We consider the Lie algebra $\frg = \frsu(2)\oplus\frsu(2)\oplus\R$ and a basis  $(e^1,\ldots,e^7)$ of $\frg^*$ 
for which the structure equations are the following  
\[
de^1=e^{23},~de^2=-e^{13},~de^3=e^{12},~de^4=e^{56},~de^5=-e^{46},~de^6=e^{45},~de^7=0. 
\]
We will study the Laplacian coflow starting at the strong $\G_2T$-structure defined by the $3$-form 
\[
\f_0 = \left( e^{14} + e^{25} - e^{36}\right)\W e^7 +  e^{123} + e^{156}-e^{246} - e^{345}.
\]
Recall from the proof of Proposition \ref{S3S3S1G2T} that this definite $3$-form induces the metric 
$g_0 = \sum_{i=1}^7 e^i\odot e^i$ and the orientation $dV_0 = e^{1234567}$. The dual $4$-form is
\[
\psi_0 = \star_0\f_0 = e^{4567} + e^{2345} + e^{2367} - e^{1346} + e^{1357} - e^{1247} - e^{1256}, 
\]
and it satisfies the equation $d\psi_0 = \theta_0\W\psi_0$,
where the Lee form is $\theta_0=e^7$. Moreover, the torsion $3$-form of $\f_0$ is $T_0=e^{123} + e^{456}$. 

The result of the previous section suggests the choice of following Ansatz for the solution of the Laplacian coflow starting at $\psi_0$
\[
\psi(t) = a_1(t)\,e^{4567} + a_2(t)\left(e^{2345} + e^{2367} - e^{1346} + e^{1357} - e^{1247} - e^{1256}\right), 
\]
where $a_1$ and $a_2$ are unknown functions such that $a_1(0)=1=a_2(0)$. 

As before, this $4$-form is definite as long as these functions are positive. 
When this happens, it defines a $\G_2T$-structure with constant Lee form $\theta(t)\equiv\theta_0 = e^7$, since $d\psi(t) = e^7 \W \psi(t)$,
and it induces the metric 
\[
g_{\psi(t)}	=	\sum_{k=1}^3 \frac{a_2}{\sqrt{a_1}}\,e^k\odot e^k  + \sum_{k=4}^7 \sqrt{a_1}\,e^k\odot e^k. 
\]
We then obtain
\[
\Delta_{\psi(t)}\psi(t) 	= \frac{a_1+a_2}{\sqrt{a_1}}\left(e^{2345} + e^{2367} - e^{1346} + e^{1357} - e^{1247} - e^{1256}\right), 
\]
whence it follows that the Laplacian coflow starting at $\psi_0$ is equivalent to the initial value problem 
\[\left\{
\begin{split}
&a_1' 		= 0, \\
&a_2' 		= \frac{a_1+a_2}{\sqrt{a_1}},\\
&a_1(0) 	= 1 = a_2(0). 
\end{split}\right.
\]
The solution is 
\[
a_1(t) \equiv 1, \quad a_2(t) = 2\exp(t)-1.  
\]
Thus, the solution to the Laplacian coflow starting at $\psi_0$ is given by the $4$-form 
\[
\psi(t) = e^{4567} + \left(2\exp(t)-1\right)\left(e^{2345} + e^{2367} - e^{1346} + e^{1357} - e^{1247} - e^{1256}\right)
\]
and it exists for all $t > -\log2$. Standard computations show that the $3$-form $\f(t) = \star_t\psi(t)$ has the following expression
\[
\f(t)  = \left(2\exp(t)-1\right)^{3/2}e^{123}+ \left(2\exp(t)-1\right)^{1/2}\left(e^{145}+e^{167}+ e^{246}-e^{257} - e^{347}-e^{356}\right), 
\]
and the intrinsic torsion form $\tau_0(t)$ is given by 
\[
\tau_0(t) = \frac17\star_t(d\f(t)\W\f(t)) = \frac67 \frac{1}{\sqrt{2\exp(t)-1}}, 
\]
so that $\lim_{t\to+\infty}\tau_0(t) = 0$. 
Finally, the torsion $3$-form is 
\[
T(t) =  \frac{7}{6} \tau_0(t)  \f(t) - \star_t d\f(t) + \star_t(\theta(t) \wedge \f(t)) = \sqrt{2\exp(t)-1}\,e^{123} + e^{456}. 
\]
Notice that $T(t)$ is closed, thus $\f(t)$ defines a strong $\G_2T$-structure. 

\smallskip

Here, for every $t\in( -\log2,+\infty)$, the global frame 
\[
e_k(t) \coloneqq  \left(2\exp(t)-1\right)^{1/2} e_k,~k=1,2,3,\quad e_k(t) \coloneqq e_k,~k=4,5,6,7,
\]
is adapted to the $\G_2$-structure $\f(t)$, and   
the manifold $\SU(2)\times\SU(2)\times\U(1)$ splits as the Riemannian product of the $3$-sphere $S^3\cong\SU(2)$ endowed with the metric 
$g_{\f(t)}|_{S^3} = \left(2\exp(t)-1\right) \sum_{k=1}^3e^k\odot e^k$ and the homogeneous Hopf surface $X^4= \SU(2)\times\U(1)$ 
endowed with the metric $g_{\f(t)}|_{X^4} = \sum_{k=4}^7e^k\odot e^k$.  
In particular, the $3$-sphere is an associative submanifold of $\SU(2)\times\SU(2)\times\U(1)$ and its volume increases exponentially along the flow, 
while the Hopf surface $X^4$ is a coassociative submanifold of $\SU(2)\times\SU(2)\times\U(1)$  and its volume is fixed along the flow.

\bigskip

{\bf Acknowledgements.}
The authors were supported by GNSAGA of INdAM and by the project PRIN 2017  
``Real and Complex Manifolds: Topology, Geometry and Holomorphic Dynamics''. 
A.F.~and A.R.~were also supported by the project PRIN 2022 "Real and Complex Manifolds: Geometry and Holomorphic Dynamics". 
A.F.~was also supported by a grant from the Simons Foundation ($\#$944448).  
The authors would like to thank Mario Garcia-Fern\'andez for useful conversations.


\begin{thebibliography}{10}

\bibitem{Agr}  
I.~Agricola.
\newblock The {S}rn\'\i\ lectures on non-integrable geometries with torsion.
\newblock {\em Arch.~Math. (Brno)} {\bf42}, 5--84, 2006. 

\bibitem{AF} 
I.~Agricola, T.~Friedrich.
\newblock A note on flat metric connections with antisymmetric torsion. 
\newblock {\em Differential Geom.~Appl.} {\bf28}, 480--487, 2010. 

\bibitem{AK}
D.~V.~Alekseevsky, B.~N.~Kimelfeld.
\newblock Structure of homogeneous Riemannian spaces with zero Ricci curvature.
\newblock {\em Funktional.~Anal.~i Prilo\v{z}en} {\bf 9}, 5--11, 1975.

\bibitem{AAF}
 L. \'Alvarez-C\'nsul, A.  De Arriba de La Hera, M. Garcia-Fernandez.
\newblock $(0,2)$ Mirror Symmetry on homogeneous Hopf surfaces.
\newblock
 \href{https://arxiv.org/abs/2012.01851}{arXiv:2012.01851}. To appear in {\em International Mathematics Research Notices}. 

\bibitem{BF}  
L. Bagaglini, A. Fino.
\newblock The Laplacian coflow on almost-abelian Lie groups.
\newblock {\em Ann.~Mat.~Pura Appl.} {\bf197}, 1855--1873, 2018. 

\bibitem{BeVe}
L.~Bedulli, L.~Vezzoni.
\newblock The {R}icci tensor of {SU}(3)-manifolds.
\newblock {\em J.\,Geom.\,Phys.} {\bf 57}, 1125--1146, 2007.

\bibitem{BerH}
M.~Berger.
\newblock Les vari\`et\`es riemanniennes homog\`enes normales simplement connexes \`a courbure strictement
positive.
\newblock{\em Ann.~Scuola Norm.~Sup.~Pisa}~\textbf{15},  179--246, 1961.

\bibitem{Bis} 
J.-M.~Bismut.
\newblock A local index theorem for non-K\"ahler manifolds,.
\newblock {\em Math.~Ann.} {\bf284}, 681--699, 1989. 

\bibitem{Bry}
R.~L.~Bryant.
\newblock Some remarks on {G$_2$}-structures.
\newblock In {\em Proceedings of {G}{\"o}kova {G}eometry-{T}opology {C}onference 2005}, pages 75--109. 
G{\"o}kova Geometry/Topology Conference (GGT), G{\"o}kova, 2006.

\bibitem{ChSa}
S..~G.~Chiossi, S.~Salamon.
\newblock The intrinsic torsion of {$\rm SU(3)$} and {G$_2$} structures.
\newblock In {\em Differential geometry, {V}alencia, 2001}, pages 115--133. World Sci. Publ., River Edge, NJ, 2002.

\bibitem{ChSw} 
S.~G. Chiossi, A.~Swann.
\newblock G{$_2$}-structures with torsion from half-integrable nilmanifolds.
\newblock {\em J.~Geom.~Phys.} {\bf 54}, 262--285, 2005.

\bibitem{CFT} 
\newblock A. Clarke, M. Garcia-Fernandez, C. Tipler.
\newblock T-Dual solutions and infinitesimal moduli of the $G_2$-Strominger system.
\newblock \href{https://arxiv.org/abs/2005.09977}{arXiv:2005.09977}. To appear in {\em Advances in Theoretical and Mathematical Physics}. 

\bibitem{FerGray} 
M.~Fern\'andez, A.~Gray. 
\newblock  Riemannian manifolds with structure group $\G_2$.
\newblock  {\em Ann.~Mat.~Pura Appl.}  {\bf 32}, 19--45, 1982.

\bibitem{Fr}
T. Friedrich.
\newblock $G_2$-manifolds with parallel characteristic torsion. 
\newblock {\em  Differential Geom.~Appl.} \textbf{25}, 632--648, 2007.

\bibitem{FI}
T. Friedrich,  S.  Ivanov.
\newblock Parallel spinors and connections with skew-symmetric torsion in string theory.
\newblock {\em Asian J.~Math.}  \textbf{6}, 303--335, 2002.

\bibitem{FI2}
T. Friedrich,  S.  Ivanov.
\newblock Killing spinor equations in dimension 7 and geometry of integrable G$_2$-manifolds. 
\newblock {\em J.~Geom.~Phys.} {\bf48}, 1--11, 2003.

\bibitem{Gar} 
M.~Garcia-Fern\'andez.
\newblock Ricci flow, Killing spinors, and T-duality in generalized geometry.
\newblock {\em Adv.~Math.} {\bf350}, 1059--1108, 2019. 

\bibitem{GRST}
M. Garcia-Fern\'andez, R. Rubio, C. S. Shahbazi, C. Tipler. 
\newblock Canonical metrics on holomorphic Courant algebroids.   
\newblock {\em Proc.~London Math.~Soc.} {\bf125}, 700--758, 2022. 

\bibitem{FS} 
M.~Garcia-Fern\'andez, J.~Streets.
\newblock Generalized Ricci Flow.
\newblock{ \em AMS University Lecture Series} {\bf 76}, 2021.

\bibitem{Gau} 
P.~Gauduchon.
\newblock Hermitian connections and {D}irac operators.
\newblock {\em Boll.~Un.~Mat.~Ital.~B} {\bf11}, 257--288, 1997. 

\bibitem{GMPW}
J. P. Gauntlett, D. Martelli, S. Pakis,  D. Waldram. 
\newblock G-Structures and Wrapped NS5-Branes. 
\newblock {\em Commun.~Math.~Phys.} {\bf247}, 421--445, 2004. 

\bibitem{Hitchin}
N.~Hitchin.
\newblock{Stable forms and special metrics.}
 Global differential geometry: the mathematical legacy of Alfred
Gray (Bilbao, 2000), Contemp.~Math., vol.~288, AMS, Providence, RI, 2001, pp.~70--89.

\bibitem{Iva} 
S. Ivanov. 
\newblock Heterotic supersymmetry, anomaly cancellation and equations of motion. 
\newblock {\em Phys.~Lett.~B} {\bf685}, 190--196, 2010.  

\bibitem{IP} 
S. Ivanov, G. Papadopoulos. 
\newblock Vanishing theorems and string backgrounds. 
\newblock {\em Classical Quantum Gravity} {\bf18}, 1089--1110, 2001. 

\bibitem{KMT} 
S.~Karigiannis, B.~McKay,  M.-P. Tsui. 
\newblock Soliton solutions for the {L}aplacian co-flow of some {$\G_2$}-structures with symmetry.
\newblock {\em Differential Geom.~Appl.} {\bf30}, 318--333, 2012.  
 
\bibitem{KL} 
A.~Kennon, J.~D.~Lotay. 
\newblock Geometric Flows of $G_2$-Structures on $3$-Sasakian $7$-Manifolds. 
\newblock {\em  J.~Geom.~Phys.} {\bf187}, 104793, 2023.   
 
\bibitem{LM}
H.~V. L\^e and M.~Munir.
\newblock Classification of compact homogeneous spaces with invariant {$\rm G_2$}-structures.
\newblock {\em Adv.~Geom.} {\bf12}, 302--328, 2012. 

\bibitem{LSS} 
J.~D.~Lotay, H.~N.~S\'a Earp, J. Saavedra. 
\newblock Flows of $G_2$-structures on contact Calabi-Yau 7-manifolds. 
\newblock {\em Ann.~Glob.~Anal.~Geom.} {\bf62}, 367--389, 2022. 


\bibitem{Mil}
J.~Milnor.
\newblock Curvatures of left invariant metrics on {L}ie groups.
\newblock {\em Adv.~Math.} {\bf21}, 293--329, 1976.

\bibitem{PS}
S. Picard, C. Suan. 
\newblock Flows of $G_2$-Structures associated to Calabi-Yau Manifolds. 
\newblock \href{https://arxiv.org/abs/2209.03411}{arXiv:2209.03411}. 

\bibitem{Reid}
F. Reidegeld.~Spaces admitting homogeneous $G_2$-structures.~{\em Differential Geom.~Appl.} {\bf 28}, 301--312,  2010. 

\bibitem{ScSc}
L.~Sch{{\"a}}fer, F. L. Schulte-Hengesbach.
\newblock Nearly pseudo-{K}{\"a}hler and nearly para-{K}{\"a}hler six-manifolds.
\newblock In {\em Handbook of pseudo-{R}iemannian geometry and supersymmetry},
  vol.~16 of {\em IRMA Lect.~Math.~Theor.~Phys.}, pp.~425--453. Eur.~Math.~Soc., Z{\"u}rich, 2010.

\bibitem{Schulte}
F. L. Schulte-Hengesbach.
\newblock Half-flat structures on Lie groups, PhD Thesis (2010), Hamburg.
\newblock  Available at \href{https://www.math.uni-hamburg.de/home/schulte-hengesbach/diss.pdf}{math.uni-hamburg.de/home/schulte-hengesbach/diss.pdf}.

\bibitem{Str} 
J. Streets.
\newblock Generalized geometry, T-duality, and renormalization group flow.
\newblock {\em J.~Geom.~Phys.} {\bf114}, 506--522, 2017. 



\end{thebibliography}
\end{document}